\documentclass[11pt]{amsart} 
\usepackage[english]{babel} 
\usepackage[T1]{fontenc} 
\usepackage{graphicx} 
\usepackage{amsmath,amsthm,amssymb} 
\usepackage{mathrsfs} 
\usepackage{enumitem} 
\usepackage[all]{xy} \usepackage{tikz-cd} 
\usepackage{xcolor} 
\usepackage{hyperref}

\DeclareMathOperator{\Aut}{Aut}
\newcommand{\idgen}{\Ch(X^\infty)_{\mathrm{id-gen}}^{[3]}}
\newcommand{\affgen}{\Ch(X^\infty)_{\mathrm{aff-gen}}^{[3]}}

\newcommand{\bdinf}{X^\infty}

\newcommand{\rk}{\mathrm{rk}}

\newcommand{\bbR}{\mathbb{R}}

\newcommand{\bbQ}{\mathbb{Q}}
\newcommand{\bbZ}{\mathbb{Z}}

\newcommand{\bbG}{\mathbb{G}}

\newcommand{\bfG}{\mathbf{G}}

\newcommand{\bfM}{\mathbf{M}}
\newcommand{\bfL}{\mathbf{L}}
\newcommand{\bfS}{\mathbf{S}}

\newcommand{\SL}{\mathrm{SL}}
\newcommand{\GL}{\mathrm{GL}}

\newcommand{\Ch}{\mathrm{ch}}
\DeclareMathOperator{\Stab}{Stab}
\DeclareMathOperator{\Span}{span}

\theoremstyle{plain}
\newtheorem{thm}{Theorem}[section]

\newtheorem{cor}[thm]{Corollary}
\newtheorem{lem}[thm]{Lemma}
\newtheorem{prop}[thm]{Proposition}
\theoremstyle{definition} 
\newtheorem{Def}[thm]{Definition}
\theoremstyle{remark}
\newtheorem{rem}[thm]{Remark}
\newtheorem{remdef}[thm]{Remark/Definition}

\newtheorem{ex}[thm]{Example}

\newtheorem*{ackn}{Acknowledgements}

\title[]{The geometry of triples of antipodal ideal chambers of affine buildings}
\author{Corina Ciobotaru} \address{Department of Mathematics, Aarhus University, Denmark} \email{cociobotaru@math.au.dk}
\author{Corentin Le Bars} \address{DMA, \'Ecole Normale Sup\'erieure, PSL University, France} \email{corentin.le.bars@math.ens.psl.eu}
\date{August 5, 2026}

\begin{document}
	\begin{abstract}
		In this article, we investigate two notions of \emph{genericity} for triples of antipodal ideal chambers in a locally finite affine building \(X\): one defined at the ideal boundary \(X^{\infty}\), which we call \emph{ideal-genericity}, and the other defined from within the affine building \(X\), which we call \emph{affine-genericity}. While ideal-genericity implies affine-genericity, the latter is the more suitable notion for constructing a barycenter map associated with affine-generic triples of antipodal ideal chambers. This perspective also allows us to establish that this barycenter map is locally constant, and hence continuous. 
		
		Finally, we provide sufficient geometric conditions on the affine Weyl group associated with $X$ that guarantee both ideal- and affine-genericity for all triples of antipodal ideal chambers of $X^\infty$. These conditions yield an algorithmic method for constructing geometric configurations in $X \cup X^{\infty}$ (when they exist) that correspond to non-generic triples of antipodal ideal chambers of $X^{\infty}$. Furthermore, our computations for the irreducible finite Weyl groups of rank at least two show that automatic ideal-genericity holds for all triples of antipodal ideal chambers only in types $B_2 = C_2$, $G_2$ and $B_3$.
	\end{abstract}
	\maketitle

	\setcounter{tocdepth}{1}
	\tableofcontents

	\section{Introduction}
	
	Affine buildings  and their spherical boundaries at infinity  play a central role in geometry, representation theory, and the structure theory of groups over non-archimedean local fields. In the authors' previous joint work \cite{CiBars} on $C^\ast$-simplicity for groups acting on affine buildings, triples of opposite (also called antipodal) ideal chambers play a significant role. Motivated by this, in the present article we investigate more closely the geometric configurations formed by those triples of pairwise opposite chambers in the spherical building at infinity $X^\infty$ of a locally finite affine building $X$. Our main focus is the interplay between two natural notions of genericity for such triples: one defined at infinity and the other defined intrinsically inside the affine building.

	More precisely, given a triple of antipodal ideal chambers
	\[
	\{C_1,C_2,C_3\}\subset \mathrm{ch}(X^\infty),
	\]
	we distinguish between \emph{ideal-genericity} (see Definition \ref{def generic}), expressed in terms of the intersections of the apartments at infinity determined by the pairs $(C_i,C_j)$, and \emph{affine-genericity} (see Definition \ref{def::generic_position_1}), formulated using the geometry of the intersections of the corresponding affine apartments inside the affine building. Although the two notions are closely related, they are not equivalent in general.
	
	The first structural result of the paper is Proposition~\ref{prop::non_generic_flats}, where we show that affine non-genericity propagates simultaneously through the entire antipodal triple. More precisely, if one of the intersections of two affine apartments has no bounded boundary component, then the same phenomenon occurs for all three intersections. As a consequence, one obtains a canonical root-flat (see Definition \ref{def::affine_flat} ), well defined up to parallelism, which governs the geometric obstruction to genericity.
	
	A second main contribution is the construction of a barycenter map for affine-generic triples. Given an affine-generic antipodal triple, we associate to it a finite convex chamber subcomplex inside the affine building and define its barycenter. Proposition~\ref{prop::barycenter_map} proves that the resulting barycenter map is locally constant, hence continuous. This construction provides a geometric invariant naturally attached to affine-generic triples and suggests new connections between asymptotic geometry and convexity phenomena in affine buildings.
	
	The second part of the paper is devoted to understanding when non-generic configurations may occur. Proposition~\ref{prop antipodal non generic case}  gives a geometric criterion in terms of the affine Coxeter complex: if there is no root-flat perpendicular to a bisector, then every antipodal triple of ideal chambers is automatically ideal-generic (and therefore affine-generic). Proposition~\ref{prop:geom_criterion}
	reformulates this condition purely in terms of finite root systems and Weyl vectors, reducing the problem to the existence of proper root subspaces containing the Weyl vector.
	
	This reduction allows us to carry out an explicit case-by-case analysis of irreducible finite Weyl groups. The outcome is summarized in Theorem~ \ref{thm:summary}, which shows that among irreducible finite Weyl groups of rank at least two, the geometric condition of Proposition~\ref{prop:geom_criterion} holds precisely in the three cases
	\[
	G_2,\qquad B_2=C_2,\qquad B_3.
	\]
	In all remaining types, there exist root-flats perpendicular to bisectors, leading to explicit constructions of non-generic configurations.
	
	Finally, we use these geometric configurations to construct examples of antipodal triples which fail to be affine-generic or ideal-generic. In the algebraic setting of Bruhat--Tits buildings, these constructions also admit an interpretation in terms of stabilizers and parabolic dynamics.
	
	\section{Basic definitions}
	\label{section::opposite_ideal_ch}
	
	For the rest of the paper, we use the following notation. If $\Delta$ denotes an affine or spherical building, then $\Ch(\Delta)$ denotes the set of all chambers of $\Delta$. If $X$ is an affine building, then $X^\infty$ denotes the spherical building at infinity of $X$, or equivalently, the ideal (also called visual) boundary of $X$.
	
	\begin{Def}
		\label{def::affine_flat} 
		Let \( X \) be an affine building, and let $A$ be an affine apartment of $X$. A subset $F \subseteq A$ is called a \textbf{root-flat} of $A$ if it is the intersection of finitely many affine walls of $A$ passing through a common point, and $F$ is not a single point.
		
		We say that $F$ is a root-flat of an affine building $X$, if there is an apartment $A$ of $X$ such that $F$ is a root-flat of $A$. Using the last axiom of buildings, together with the fact that walls extend across apartments, one observes that the root-flat $F$ is independent of the choice of apartment $A$ in $X$.
	\end{Def}		
	
	Notice that being a root-flat in an affine apartment $A$ is equivalent to being an unbounded intersection of a set of affine walls of $A$. Indeed, the intersection of walls in an apartment through a point is an affine subspace, so it is either a point or unbounded.

	\begin{remdef}[Finite and infinite boundary of the intersection of two affine apartments]
			\label{rem::bounded_cone_affine_flats}
			Let \(X\) be an affine building, and let \(A_1,A_2\) be two distinct affine apartments such that $\Ch(A_1^\infty)\cap \Ch(A_2^\infty)\neq \emptyset$, and set $Q_{12}:=A_1\cap A_2$. The set \(Q_{12}\) is a convex polyhedron in the affine Coxeter complex \(A_1\), i.e. there is a finite family of affine half-apartments \(D_1,\ldots,D_N\) of \(A_1\) such that
			\[
			Q_{12}=\bigcap_{\nu=1}^N D_\nu .
			\]
			This description is not unique, but it determines a canonical finite face poset \(\mathscr F(Q_{12})\), namely the usual face poset of the convex polyhedron \(Q_{12}\). We write
			\[
			\mathscr F_\partial(Q_{12})
			:=
			\mathscr F(Q_{12})\setminus\{Q_{12}\}
			\]
			for the set of proper non-empty faces. These are precisely the faces contained in the relative boundary \(\partial Q_{12}\) of \(Q_{12}\) in $X$. 
			
			We define the finite boundary part of \(Q_{12}\) by
			\[
			\partial Q_{12}^{\mathrm{fin}}
			:=
			\bigcup \{F \mid \ F\in \mathscr F_\partial(Q_{12}), \
			F \text{ bounded}\},
			\]
			and the infinite boundary part by
			\[
			\partial Q_{12}^{\mathrm{inf}}: = \partial Q_{12}\setminus \partial Q_{12}^{\mathrm{fin}}.
			\]
		\end{remdef}

		Observe that each face of the polyhedron \(Q_{12}\) is contained in an intersection of affine walls of \(A_1\). In particular, the affine span of every positive-dimensional face is a root-flat of \(A_1\). The same definition is obtained if \(Q_{12}\) is viewed inside \(A_2\), because any apartment isomorphism \(A_1\to A_2\) fixing \(A_1\cap A_2\) pointwise preserves the face poset of \(Q_{12}\).
		
		Observe also that at least one of the sets \(\partial Q_{12}^{\mathrm{fin}}\) and \(\partial Q_{12}^{\mathrm{inf}}\) is non-empty. For example, if \(Q_{12}\) is a half-apartment, then its boundary is one unbounded face, and therefore \(\partial Q_{12}^{\mathrm{fin}}=\emptyset\). If \(Q_{12}\) is a sector, then the vertex at its tip is a bounded face, so \(\partial Q_{12}^{\mathrm{fin}}\neq\emptyset\).

	\begin{lem}
			\label{lem::affine_flat}
			Let \(A_1,A_2\) be two distinct affine apartments of an affine building \(X\), and suppose that
			\(Q_{12}:=A_1\cap A_2\) contains a sector. If
			\[
			\partial Q_{12}^{\mathrm{fin}}=\emptyset,
			\]
			then \(\partial Q_{12}^{\mathrm{inf}}\) contains a root-flat of both \(A_1\)  and \(A_2\).
		\end{lem}
		
		\begin{proof}
			As before we view \(Q_{12}\) as a closed convex polyhedron in \(A_1\), and consider its	finite face poset \(\mathscr F(Q_{12})\). Since \(A_1\neq A_2\), the set \(Q_{12}\) is a proper convex subset of \(A_1\), hence \(\mathscr F_\partial(Q_{12})\neq\emptyset\). Choose a face
			\[
			F\in \mathscr F_\partial(Q_{12})
			\]
			of minimal possible dimension.
			
			By assumption, \(Q_{12}\) has no bounded proper face, hence \(F\) is unbounded. We claim that \(F\) is an affine subspace. Indeed, otherwise it would have a non-empty proper face. Such a face would also be a proper face of \(Q_{12}\), contradicting the minimality of \(\dim F\). Thus \(F\) is indeed an affine subspace.
			
			Since \(F\) is a face of the polyhedron \(Q_{12}\), it is obtained as an intersection of affine walls of \(A_1\). Since \(F\) is unbounded and an affine subspace, it has positive dimension, and hence it is a root-flat of \(A_1\). Moreover \(F\subseteq \partial Q_{12}^{\mathrm{inf}}\).
			
			Finally, an apartment isomorphism \(A_1\to A_2\) fixing \(A_1\cap A_2=Q_{12}\) pointwise sends affine walls of \(A_1\) to affine walls of \(A_2\), and fixes \(F\) pointwise. Hence \(F\) is also an intersection of affine walls of \(A_2\), and thus \(F\) is a root-flat in \(A_2\) as well.
	\end{proof}
	
	There are two natural approaches to defining generic triples of ideal chambers: one is from the viewpoint at infinity, and the other is from the perspective inside the affine building.
	
	\begin{Def}\label{def generic}
		Let \( X \) be a locally finite affine building, and let \( X^\infty \) denote its spherical boundary at infinity. We say that a triple of distinct ideal chambers \( \{C_1, C_2, C_3\} \subset \Ch(X^\infty) \) is \textbf{antipodal} if the chambers \( C_1, C_2, C_3 \) are pairwise opposite.
		
		We say that the triple \( \{C_1, C_2, C_3\} \) is in \textbf{ideal-generic position} if it is antipodal and, for all \( 1 \leq i \neq j \leq 3 \), denoting by \( A_{ij} \) the unique apartment in $X$ having both  \( C_i \) and \( C_j \) as chambers in its boundary at infinity \(A_{ij}^{\infty} \), we have:
		\[
		A_{12}^{\infty} \cap  A_{13}^{\infty} \cap A_{23}^{\infty} = \emptyset.
		\]
		Let us denote by $\Ch(X^\infty)^{[3]}$ the set of triples of ideal chambers of $X$ that are antipodal and by $\idgen$ those that are in ideal-generic position.
	\end{Def}

	\begin{Def}
		\label{def::generic_position_1}
		Let \( X \) be a locally finite affine building, and let  \(\{C_1, C_2, C_3\}\)
		be an antipodal triple of ideal chambers in  $\Ch(X^\infty)$. Denote by \( A_{ij} \) the unique apartment in \( X \) whose boundary at infinity \( A_{ij}^{\infty} \) contains both \( C_i \) and \( C_j \) as chambers. For \(\{i,j,k\}=\{1,2,3\}\), define
		\[
		Q_i := A_{ij} \cap A_{ik}.
		\]
		Using the notation from Remark~\ref{rem::bounded_cone_affine_flats}, we say that the triple $\{C_1, C_2, C_3\}$ is in \textbf{affine-generic position} if 
		\[
		\partial Q_i^{\mathrm{fin}} \neq \emptyset,
		\]
		for every $i \in \{1,2,3\}$.
		Let us denote by $\affgen$ the set of antipodal triples of ideal chambers in $\Ch(X^\infty)$ that are in affine-generic position.
	\end{Def}

	\section{Sub-buildings arising from root-flats}
	\label{sec::sub_buildings_affine_flats}
	
	Let \( X \) be a locally finite affine building. Let \( \Sigma \) be an apartment of \( X \), and let \( \eta_{+} \) and \( \eta_{-} \) be two opposite ideal simplices in \( \Sigma^\infty \) that are fixed for the rest of this section. Let \( (W_I, I) \) denote the subgroup of the finite Weyl group \( (W_{\mathrm{fin}}, S_{\mathrm{fin}}) \) associated with \( X \) that stabilizes the simplex \( \eta_{+} \). The \textbf{residue} \( \operatorname{res}(\eta_{+}) \) is defined as the set of all ideal chambers in \( \Ch(X^\infty) \) that contain \( \eta_{+} \). Let \( X(\eta_{+}, \eta_{-}) \) denote the union of all apartments in \( X \) whose ideal boundaries contain both \( \eta_{+} \) and \( \eta_{-} \). Let \( n \) denote the dimension of \( X \).
	
	\begin{prop}[Section 4.3 in \cite{Rousseau2011}]
		\label{prop::res_building}
		With the notation introduced above, the union of apartments \( X(\eta_{+}, \eta_{-}) \) forms a closed, convex subset of \( X \). It is an extended, locally finite, thick affine building, and satisfies
		\[
		X(\eta_{+}, \eta_{-}) \cong \mathbb{R}^{n - |I|} \times X_I,
		\]
		where \( X_I \) is a locally finite affine building of dimension \( |I| \) and of finite Weyl group $(W_I,I)$. Moreover, the residue \( \operatorname{res}(\eta_{+}) \) is a compact subset of \( \Ch(X^\infty) \). It is a spherical building, and
		\[
		\operatorname{res}(\eta_{+}) \cong \Ch(X_I^\infty).
		\]
	\end{prop}
	
	When \( \eta_{+} \) and \( \eta_{-} \) are two opposite chambers in \(\Sigma^\infty\), we have \( X(\eta_{+}, \eta_{-}) \cong \mathbb{R}^{n} \cong \Sigma \), and \( X_I \) reduces to a point. Note that the Euclidean factor \( \mathbb{R}^{n - |I|} \) in the splitting of \( X(\eta_{+}, \eta_{-}) \) contains the simplices \( \eta_{+} \) and \( \eta_{-} \) in its boundary at infinity. Moreover, any flat of minimal dimension in \( X \) whose ideal boundary contains \( \eta_{+} \) and \( \eta_{-} \) is isomorphic to \( \mathbb{R}^{n - |I|} \).
	
	\begin{rem}
		\label{rem::apartments_panel_building}
		From the construction of $X(\eta_{+},\eta_{-}) \cong \mathbb{R}^{\,n - |I|} \times X_I$, it follows that the affine apartments of $X_I$  are exactly in bijection with the apartments of $X$ whose ideal boundaries contain both $\eta_+, \eta_-$ (thus the apartments in $X(\eta_{+},\eta_{-})$. Moreover, it follows that the affine walls of $X_I$ are in bijection with the projections of the affine walls of $X$ (i.e. after quotienting out the $\mathbb{R}^{\,n - |I|}$ factor), and thus of $X(\eta_{+},\eta_{-})$, whose ideal boundaries contain $\eta_{+}$ and $\eta_{-}$. The same holds true for root-flats in $X_I$.
	\end{rem}
	
	\begin{lem}
		\label{lem::opposition_projection}
		With the notation fixed in this section, let $A$ be an apartment of
		$X(\eta_{+},\eta_{-})$, and let $C_{+}, C_{-} \in \Ch(A^{\infty})$ be two opposite
		ideal chambers. Consider the projection
		\[
		p_I : X(\eta_{+},\eta_{-})
		\cong \mathbb{R}^{\,n - |I|} \times X_{I}
		\longrightarrow
		X_{I},
		\]
		and set $\mathcal{A} := p_I(A)$.  Then:
		\begin{enumerate}
			\item The projection $p_I(A)=:\mathcal A$ is an apartment of $X_I$, and $(p_I(A))^\infty=:\mathcal A^\infty$ is an apartment of $X_I^\infty$.
			\item The interiors of $C_+$ and $C_-$ project by the map $p_I$ to the interiors of two chambers $\mathcal C_+,\mathcal C_-\in\Ch(\mathcal A^\infty)$.
			\item These two chambers $\mathcal C_+$ and $\mathcal C_-$ are opposite in $X_I^\infty$.
			\item For any chambers $\mathcal C_1,\mathcal C_2\in\Ch(X_I^\infty)$, there exists an apartment $A_{12}$ of $X(\eta_+,\eta_-)$ such that $p_I(A_{12})$ is an apartment of $X_I$ whose ideal boundary contains $\mathcal C_1$ and $\mathcal C_2$.  If $\mathcal C_1$ and $\mathcal C_2$ are opposite, then $A_{12}$ is unique.
			\item More generally, the interiors of two opposite chambers of $X(\eta_+,\eta_-)^\infty$ project by the map $p_I$ to the interiors of two opposite chambers in $\Ch(X_I^\infty)$.
		\end{enumerate}
	\end{lem}
	
	\begin{proof}
		By Proposition~\ref{prop::res_building} and Remark~\ref{rem::apartments_panel_building}, every apartment of $X(\eta_+,\eta_-)$ is of the form
		\[
		A=\mathbb{R}^{\,n - |I|}\times\mathcal A
		\]
		for a unique affine apartment $\mathcal A$ of $X_I$, and the projection $p_I:\mathbb{R}^{\,n - |I|}\times X_I\to X_I$ sends $A$ onto $\mathcal A$. Hence $p_I(A)=\mathcal A$ and $(p_I(A))^\infty=\mathcal A^\infty$. This proves (1).
		
		Let us prove (2) and (3).  From the same remark, we recall that in the apartment \(A\) of \(X(\eta_{+},\eta_{-})\), we retain only those affine walls of \(A\) whose ideal boundaries contain both \(\eta_{+}\) and \(\eta_{-}\). These are precisely the walls that project to the affine walls in the apartment \(\mathcal{A}\) of the affine building \(X_I\).
		
		Moreover, in the spherical apartment \(A^\infty\), any two opposite chambers are separated by the walls of that spherical apartment. Viewing this configuration from the corresponding affine apartment \(A\), the two opposite ideal chambers \(C_{+}\) and \(C_{-}\) are therefore separated, in particular, by each affine wall of \(A\) whose ideal boundary contains \(\eta_{+}\) and \(\eta_{-}\).
		
		In addition, let \(F\) be any root-flat in \(A\) that contains \(\eta_{+}\) and \(\eta_{-}\) in its ideal boundary and has minimal
		dimension with this property. Since \(F\) is, by definition, a root-flat and \(X\) is locally finite, the number of affine walls of \(A\)
		containing \(F\) is finite. Moreover, these affine walls, whose collection is denoted by \(\mathcal{H}(A,F)\), partition the apartment \(A\) into a finite collection of convex cones, all of which intersect in the affine space \(F\).
		
		Each ideal chamber of \(A^\infty\) lies in the ideal boundary of exactly one of these cones. Because \(C_{+}\) and \(C_{-}\) are opposite by assumption, the cones whose ideal boundaries contain \(C_{+}\) and \(C_{-}\) must themselves be opposite. Indeed, these cones must be defined by the same affine walls of \(\mathcal{H}(A,F)\) in order to separate the opposite ideal chambers \(C_{+}\) and \(C_{-}\).
		
		Projecting the cones in \(A \cup A^{\infty}\) determined by the family \(\mathcal{H}(A,F)\), as well as the affine walls in \(\mathcal{H}(A,F)\) themselves, to \(\mathcal{A} \cup \mathcal{A}^{\infty}\) (that is,
		quotienting out the factor \(F \cong \mathbb{R}^{n - |I|}\)) preserves the property of being an affine wall or a root-flat, and, as a consequence, preserves opposition as well.  
		
		We observe that the family \(\mathcal{H}(A,F)\) does not contain a representative from every parallelism class of affine walls in \(A\). This is because the finite Weyl group \((W_I, I)\) associated with \(X_I\) is a subgroup of the finite Weyl group \((W_{\mathrm{fin}}, S_{\mathrm{fin}})\) of \(X\), where \(I \subsetneq S\). On the one hand, this implies that each facet (i.e., each codimension-one simplex) of an ideal chamber of \(\mathcal{A}^{\infty}\) arises as the projection obtained by quotienting out the factor \(F \cong \mathbb{R}^{n - |I|}\) of some affine wall in \(\mathcal{H}(A,F)\). On the other hand, it also implies that the cones in \(A\) determined by \(\mathcal{H}(A,F)\) are wider than the usual Weyl sectors in \(X\) coming from $(W_{\mathrm{fin}}, S_{\mathrm{fin}})$.  
		
		Consequently, when projecting \(C_{+}\) and \(C_{-}\), together with their cones determined by \(\mathcal{H}(A,F)\) that uniquely contain them, respectively, to the ideal boundary \(X_I^{\infty}\) of \(X_I\), their images lie, in the interiors of, and respectively give exactly, two ideal chambers \(\mathcal{C}_{+}\) and \(\mathcal{C}_{-}\) of \(\mathcal{A}^{\infty}\), respectively. Moreover, since the cones determined by \(\mathcal{H}(A,F)\) that uniquely contain \(C_{+}\) and \(C_{-}\), respectively, are opposite, so should be the two ideal chambers \(\mathcal{C}_{+}\) and \(\mathcal{C}_{-}\) in \(\mathcal{A}^{\infty}\), thus in $X_I^{\infty}$. Thus (2) and (3) are proven. The same arguments prove (5), after choosing an apartment of $X(\eta_+,\eta_-)$ containing the two projected chambers.

%Finally, let $\mathcal C_1,\mathcal C_2\in\Ch(X_I^\infty)$. Since $X_I^\infty$ is a spherical building, there is an apartment $\mathcal A_{12}^\infty\subset X_I^\infty$ containing both $\mathcal C_1,\mathcal C_2$. Equivalently, there is an affine apartment $\mathcal A_{12}\subset X_I$ with $\mathcal C_1,\mathcal C_2\in\Ch(\mathcal A_{12}^\infty)$. Then, by  Remark~\ref{rem::apartments_panel_building}, $A_{12}:=\mathbb{R}^{\,n - |I|}\times\mathcal A_{12}$ is an apartment of $X(\eta_+,\eta_-)$, proving existence. If $\mathcal C_1$ and $\mathcal C_2$ are opposite, the apartment $\mathcal A_{12}$ is unique, and the bijection $\mathcal A\mapsto E\times\mathcal A$ gives uniqueness of $A_{12}$. This proves (4).
		
		Let us now prove (4), which serves as a converse to parts (2) and (3).  
		Since $X_{I}$ is an affine building, standard results guarantee the existence of an apartment $\mathcal{A}_{12}$ in $X_{I}$ whose ideal boundary $\mathcal{A}_{12}^{\infty}$ contains the chambers $\mathcal{C}_{1}$ and $\mathcal{C}_{2}$.  
		Define
		\[
		A_{12} := \mathbb{R}^{\,n - |I|} \times \mathcal{A}_{12}.
		\]
		By Remark~\ref{rem::apartments_panel_building}, this is an apartment of 
		$X(\eta_{+},\eta_{-})$.
		
		Inside the apartment $\mathcal{A}_{12}$, choose Weyl sectors  $\mathcal{Q}_{1}$ and $\mathcal{Q}_{2}$ based at a common vertex $x \in \mathcal{A}_{12}$ and having ideal boundaries $\mathcal{C}_{1}$ and $\mathcal{C}_{2}$, respectively.  
		Then the sets
		\[
		Q_{1} := \mathbb{R}^{\,n - |I|} \times \mathcal{Q}_{1},
		\qquad
		Q_{2} := \mathbb{R}^{\,n - |I|} \times \mathcal{Q}_{2}
		\]
		are cones in the apartment $A_{12}$, and
		\(
		F := \mathbb{R}^{\,n - |I|} \times \{x\}
		\)
		is a root-flat of $A_{12}$ whose ideal boundary contains 
		$\eta_{+}$ and $\eta_{-}$.  Moreover, the cones $Q_{1}$ and $Q_{2}$ are precisely the cones determined by the family $\mathcal{H}(A,F)$ (as defined in the proof of parts (2) and (3)), which separates the apartment $A_{12}$. 
		
		As in the proof of parts (2) and (3), the fact that the chambers $\mathcal{C}_{1}$ and $\mathcal{C}_{2}$ are opposite - and hence that the corresponding Weyl sectors $\mathcal{Q}_{1}$ and $\mathcal{Q}_{2}$ based at the same point $x$ are opposite - is equivalent to the opposition of the cones $Q_{1}$ and $Q_{2}$ in the apartment $A_{12}$. In particular, this is equivalent to the opposition of the ideal boundaries of $Q_1$ and $Q_2$. The uniqueness of the apartment $A_{12}$ follows immediately from the uniqueness of the apartment $\mathcal{A}_{12}$ in $X_{I}$ when $\mathcal{C}_{1}$ and 
		$\mathcal{C}_{2}$ are opposite, together with the way $A_{12}$ is constructed from $\mathcal{A}_{12}$. Part (4) is proven, and so is the lemma.
	\end{proof}

	\section{Non-generic antipodal triples and their associated root-flats}
	In the next proposition, we say that a root-flat \( F_1 \) is parallel to a root-flat \( F_2 \) if \( \dim F_1 \leq \dim F_2 \) and there exists a root-flat in \( F_2 \) whose ideal boundary is exactly the ideal boundary of \( F_1 \).
	
	\begin{prop}
		\label{prop::non_generic_flats}
		Let $X$ be a locally finite affine building, and let $\{C_1, C_2, C_3\}$ be an antipodal triple of chambers in $X^\infty$.
		Suppose that there exists $i \in \{1,2,3\}$ such that  $\partial Q_{i}^{\mathrm{fin}} = \emptyset$.
		Then for each $j \in \{1,2,3\}$ we have $\partial Q_{j}^{\mathrm{fin}} = \emptyset$.
		Moreover, there exists a root-flat of $X$ that is parallel to a root-flat in $\partial Q_{k}^{\mathrm{inf}}$ for every $k \in \{1,2,3\}$, and
		\[
		A_{12}^{\infty} \cap A_{13}^{\infty} \cap A_{23}^{\infty} \neq \emptyset.
		\]
	\end{prop}

	\begin{proof}
		Assume without loss of generality that $i = 1$ and that $\partial Q_{1}^{\mathrm{fin}} = \emptyset$. By applying Lemma~\ref{lem::affine_flat} to the apartments $A_{12}$ and $A_{13}$, we deduce that $\partial Q_{1}^{\mathrm{inf}}$ contains a root-flat.  Choose such a root-flat of maximal dimension and denote it by $F_1$.  Moreover, we may assume that the dimension of $F_1$ is maximal among the dimensions of all root-flats contained in $\partial Q_{i}^{\mathrm{inf}}$ 
		for any $i \in \{1,2,3\}$.
		
We claim that any other root-flat of maximal dimension contained in \(\partial Q_{1}^{\mathrm{inf}}\), if such a root-flat exists, is parallel to \(F_1\) (in the sense defined just before the proposition). Indeed, let \(F_2\) be such a root-flat. Since both \(F_1\) and \(F_2\) are root-flats contained in \(\partial Q_{1}^{\mathrm{inf}}\), they cannot intersect. Indeed, if they did, then by the geodesic convexity of \(Q_1\) and of the unbounded components of \(\partial Q_{1}^{\mathrm{inf}}\), either they would generate a root-flat of strictly larger dimension, contradicting the maximality assumption of $F_1$, or they would fail to remain root-flats contained in \(\partial Q_{1}^{\mathrm{inf}}\).

Consider now the finite set $\mathcal{W}$ of all affine walls in $A_{12}$ such that
\[
F_1=\bigcap_{H\in\mathcal{W}} H.
\]
For each $H\in\mathcal{W}$, let $D_H$ be the half-apartment of $A_{12}$ containing $Q_1 \cup \partial Q_1$. Then $F_2 \subset Q_1 \cup \partial Q_1 \subset \bigcap_{H\in\mathcal{W}} D_H$. Hence, $F_2$ is contained in $D_H$ for every affine wall $H\in\mathcal{W}$. It follows from standard Euclidean geometry that $F_2$ is parallel to $H$ for every $H\in\mathcal{W}$. Since we are working in a Euclidean space, $F_1$ and $F_2$ are disjoint, and since
\[
F_1=\bigcap_{H\in\mathcal{W}} H,
\]
we conclude that $F_2$ is parallel to $F_1$.

		Given the root-flat $F_1$ in $\partial Q_{1}^{\mathrm{inf}}$ of dimension $m$, chosen as at the beginning of the proof, we may apply 
		Proposition~\ref{prop::res_building} to the set $X(F_1)$ of all affine apartments of $X$ that contain root-flats parallel to $F_1$ and of the 
		same dimension as $F_1$. Then
		\[
		X(F_1) \cong \mathbb{R}^m \times X_{F_1},
		\]
		where $\mathbb{R}^m \cong F_1$ and $X_{F_1}$ is a locally finite affine building of dimension $n - m$. One may view $X_{F_1}$ as a transversal section of $X(F_1)$ that is ``perpendicular'' to the root-flats parallel to $F_1$.
		
		Notice that $A_{12}$ and $A_{13}$ are apartments of $X(F_1)$, and therefore determine two affine apartments $\widetilde{A}_{12}$ and $\widetilde{A}_{13}$ of $X_{F_1}$. Moreover, the affine walls of $X_{F_1}$ arise from the affine walls of $X$ that contain root-flats parallel to $F_1$ and of the same dimension as $F_1$.
		
		Since $C_1, C_2 \in \Ch(A_{12}^{\infty})$ and $C_1, C_3 \in \Ch(A_{13}^{\infty})$, their projections to the transversal section $X_{F_1} \cup X_{F_1}^{\infty}$ yield subsets $\widetilde{C}_1, \widetilde{C}_2, \widetilde{C}_3$ of three ideal chambers of $X_{F_1}^{\infty}$, denoted $\mathcal{C}_1, \mathcal{C}_2, \mathcal{C}_3$. These subsets $\widetilde{C}_1, \widetilde{C}_2, \widetilde{C}_3$ have the same dimensions as the ideal chambers  $\mathcal{C}_1, \mathcal{C}_2, \mathcal{C}_3 \in \Ch(X_{F_1}^{\infty})$. Moreover, by Lemma \ref{lem::opposition_projection}(2) and (3), the chambers 
		$\mathcal{C}_1, \mathcal{C}_2, \mathcal{C}_3$ still form an antipodal triple in $ \Ch(X_{F_1}^{\infty})$.
		
		This implies that there exist unique apartments $\mathcal{A}_{12}, \mathcal{A}_{13}, \mathcal{A}_{23}$ in $X_{F_1}$ such that 
		$\mathcal{C}_i, \mathcal{C}_j \in \Ch(\mathcal{A}_{ij})$ for every $i \neq j \in \{1,2,3\}$. Define
		\[
		\mathcal{Q}_i := \mathcal{A}_{ij} \cap \mathcal{A}_{ik},
		\qquad \text{for } i \neq j \neq k \in \{1,2,3\}.
		\]
		By the choice of $F_1$ as having maximal dimension among all possible root-flats contained in $\partial Q_1$, we must have
		\[
		\partial \mathcal{Q}_1^{\mathrm{fin}} \neq \emptyset.
		\]
		
		Moreover, by the fact that $\mathcal{C}_1, \mathcal{C}_2, \mathcal{C}_3$ are pairwise opposite, we obtain
		\[
		\mathcal{A}_{12} = \widetilde{A}_{12} \qquad \text{and} \qquad 
		\mathcal{A}_{13} = \widetilde{A}_{13}.
		\]
		Furthermore, taking the product $\mathbb{R}^m \times \mathcal{A}_{23}$ yields a unique apartment in $X$, whose ideal boundary must contain $C_2$ and $C_3$ (see Lemma \ref{lem::opposition_projection}(4)). Thus
		\[
		\mathbb{R}^m \times \mathcal{A}_{23} = A_{23},
		\]
		and therefore $A_{23}$ contains a root-flat parallel to $F_1$. By construction, we also have
		\[
		\mathbb{R}^m \times \mathcal{Q}_j = Q_j, 
		\qquad \text{for every } j \in \{1,2,3\}.
		\]
		
		Finally, by the choice of $F_1$ as having maximal dimension among all possible root-flats contained in $\partial Q_i$, we must have
		\[
		\partial \mathcal{Q}_j^{\mathrm{fin}} \neq \emptyset,
		\qquad \text{for every } j \in \{1,2,3\}.
		\]
		From here the conclusion of the proposition follows easily.
	\end{proof}

	\begin{rem}
		\label{rem::ideal_affine}
		We observe that Proposition~\ref{prop::non_generic_flats} implies that \(\Ch(X^{\infty})_{\mathrm{id\text{-}gen}}^{[3]}\) is contained in \(\Ch(X^{\infty})_{\mathrm{aff\text{-}gen}}^{[3]}\).  In general, the converse inclusion does not hold, and in Section~\ref{subsec::conf_non_affine_generic} we construct explicit examples illustrating this failure.
	\end{rem}
	
	\section{The barycenter map of affine-generic triples}
	We want to associate a well-defined barycenter map for any triple of ideal chambers that are in affine-generic position. For that we need more definitions.
	
	\begin{Def}
		\label{def::generic_position_convex_hull}
		Let  $\{C_1, C_2, C_3\} \in \Ch(X^\infty)_{\mathrm{aff\text{-}gen}}^{[3]}$.  Using the notation from Definition~\ref{def::generic_position_1}, for each \( i \in \{1,2,3\} \), let  
		\[
		\mathcal{R}_{X}(\partial Q_i^{\mathrm{fin}})
		\]
		be the set of chambers of \(\Ch(X)\) that contain a subsimplex in \(\partial Q_i^{\mathrm{fin}}\). Since \( X \) is locally finite, the set \(\mathcal{R}_{X}(\partial Q_i^{\mathrm{fin}})\) is finite. 
		
		We denote by  
		\[
		\Gamma(C_1, C_2, C_3)
		\]
		the \emph{convex hull} (i.e., the smallest convex chamber subcomplex) in \( X \) of the finite set of chambers  $\bigcup_{i \in \{1,2,3\}} \mathcal{R}_{X}(\partial Q_i^{\mathrm{fin}})$ (see \cite[Examples 3.133.(c)]{abramenko_brown08}). Since each \(\mathcal{R}_{X}(\partial Q_i^{\mathrm{fin}})\) is finite, it follows that \(\Gamma(C_1, C_2, C_3)\) is a finite collection of chambers in \( X \). Therefore, we have a well-defined map  
		\[
		\Gamma : \affgen \longrightarrow \{\text{bounded convex chamber subcomplexes of } X\}.
		\]
		
		We further define  
		\[
		\zeta : \affgen \longrightarrow X
		\]
		to be the map that associates to each  $\{C_1, C_2, C_3\} \in \Ch(X^\infty)_{\mathrm{aff-gen}}^{[3]}$ the point  $
		\zeta(C_1, C_2, C_3)$ in \( X \), which is the circumcenter of the bounded convex hull \(\Gamma(C_1, C_2, C_3)\) (see \cite[Chapter II.2, Proposition 2.7]{bridson_haefliger99}).
	\end{Def}

	\begin{prop}\label{prop::barycenter_map}
		The barycenter map \[\zeta: \affgen\to X \]
		from affine-generic triples of chambers at infinity to $X$ is locally constant, hence continuous. 
	\end{prop}
	\begin{proof}
		Let us fix  an affine-generic triple of chambers at infinity $(C_1, C_2, C_3) \in \Ch(X^\infty)^{[3]}_{\mathrm{aff\text{-}gen}}$. We use the same notation as in Definition~\ref{def::generic_position_1}. So let $A_{ij}$ be the unique affine apartment in $X$ associated with the pair of opposite ideal chambers $(C_i, C_j)$.  Let $\xi_i \in X^\infty$ be the barycenter of the ideal chamber $C_i$, where $i \in \{1,2,3\}$. For each $i \in \{1,2,3\}$, choose a point $x_{i} \in \partial Q_i^{\mathrm{fin}}  \neq \emptyset$ and for each $j \in \{1,2,3\}$, with $j \neq i$, consider the bi-infinite geodesic line $(\xi_i, x_i, \xi_j)$ in the apartment $A_{ij}$ between the ideal points $\xi_i$ and $\xi_j$. 
		
		Let $r > 0$ be a real number, which we may assume to be arbitrarily large.  Denote by $x_{r,i}$ the unique point on the geodesic ray $(\xi_i, x_i)$ that is at distance exactly $r$ from $x_i$, for each $i \in \{1,2,3\}$. Since the geodesic lines $(\xi_i, x_i, \xi_j)$ are strongly regular by construction, they are not contained in any of the walls of the affine apartment $A_{ij}$. Thus, if we perturb \( r \) a little bit, we can assume that the point \( x_{r,i} \) lies in the interior of a chamber \( c_{i} \) of \( A_{ij} \cap A_{ik} \), for $\{i,j,k\} = \{1,2,3\}$.
		
		Now, we consider the convex hull $\Gamma(c_i, c_j) \subset A_{ij}$ (see \cite[Definition 4.116]{abramenko_brown08}), for every pairwise distinct $i,j \in \{1,2,3\}$. Then, by \cite[the remarks after Definition 4.116]{abramenko_brown08}, the convex hull $\Gamma(c_i, c_j)$ lies in any affine apartment of $X$ that contains the two chambers $c_i, c_j$. Since, for each \(i \in \{1,2,3\}\), the point \(x_i\) lies in the bounded set \(\partial Q_i^{\mathrm{fin}}\), we can choose \(r\) sufficiently large so that \(\partial Q_i^{\mathrm{fin}}\) and \(\partial Q_j^{\mathrm{fin}}\) are contained in the interior of \(\Gamma(c_i,c_j)\). Moreover, \(r\) can be chosen so that, for every pairwise distinct \(i,j,k \in \{1,2,3\}\), the geodesic-ray extensions in the apartments \(A_{ij}\) and \(A_{ik}\) of \[ \Gamma(c_i,c_j)\cap\Gamma(c_i,c_k) \subset A_{ij}\cap A_{ik} = Q_i \] coincide with the cone \(Q_i\). This implies that
\begin{equation}
\label{equ::cone_in_apartment}
 Q_i\cap\Gamma(c_i,c_j) = \Gamma(c_i,c_j)\cap\Gamma(c_i,c_k).
\end{equation}
		Then, for the chosen $r$, we take $U_i:= U(x_i,r,C_i)$ to be the standard open neighborhood with base point $x_i$ and radius $r$ around $C_i$.  Then, for any $C'_i \in U_i$ with $i \in \{1,2,3\}$, and since $x_{r,i}$ lies in the interior of $c_i$, we obtain that $\Gamma(c_i, c_j)$ is contained in the unique apartment $A'_{ij}$ whose ideal boundary contains the ideal chambers $C'_i$ and $C'_j$, for all $i \neq j \in \{1,2,3\}$. Indeed, this follows from the fact that the strongly regular bi-infinite geodesic line $(\xi'_i, x_i, \xi'_j)$, where $\xi'_i$ and $\xi'_j$ are the barycenters of the ideal chambers $C'_i$ and $C'_j$, respectively, uniquely determines the apartment $A'_{ij}$, and intersects the strongly regular bi-infinite geodesic line $(\xi_i, x_i, \xi_j)$ in a long geodesic segment that determines $\Gamma(c_i, c_j)$. 
		
		By construction, we notice that $ \partial Q_i^{\mathrm{fin}}  \subset \Gamma(c_i, c_j)\cap \Gamma(c_i, c_k) \subset Q_i$, for any $i \in \{1,2,3\}$ with $\{i,j,k\}=\{1,2,3\}$. Recall also that, for every \(i \neq j \in \{1,2,3\}\), the apartments \(A_{ij}\) and \(A'_{ij}\) both contain the convex hull \(\Gamma(c_i,c_j)\). Consequently, their intersection \(A_{ij} \cap A'_{ij}\) is a convex set, possibly larger than \(\Gamma(c_i,c_j)\), that contains \(\Gamma(c_i,c_j)\). Upon considering the intersection \[ A'_{ij}\cap A'_{ik}=:Q'_i, \] we claim that 
		\[ Q'_i\cap\Gamma(c_i,c_j) = Q_i\cap\Gamma(c_i,c_j) = \Gamma(c_i,c_j)\cap\Gamma(c_i,c_k). \] 

Indeed, the second equality is just equation (\ref{equ::cone_in_apartment}). For the first equality observe that, by construction, \( \Gamma(c_i,c_j)\cap\Gamma(c_i,c_k) \) is contained in both $A_{ij}\cap A_{ik}$ and $A'_{ij}\cap A'_{ik}$. Thus, \[ Q_i\cap\Gamma(c_i,c_j)\subset Q'_i\cap\Gamma(c_i,c_j). \] Now, suppose that $Q'_i\cap\Gamma(c_i,c_j)$ were strictly larger than $\Gamma(c_i,c_j)\cap\Gamma(c_i,c_k)$. Then, by the convexity of the cone $Q'_i$, the apartment $A'_{ik}$ would contain not only $\Gamma(c_i,c_k)$, by construction, but also points of $Q'_i\cap\Gamma(c_i,c_j)$ that do not belong to $\Gamma(c_i,c_k)$ and lie in a small neighborhood of \[ \partial Q_i^{\mathrm{fin}}\subset \Gamma(c_i,c_j)\cap\Gamma(c_i,c_k) \] within \[ \Gamma(c_i,c_j)\subset A_{ij}\cap A'_{ij}. \] However, the existence of such points contradicts the fact that $\Gamma(c_i,c_k)$ is a bounded convex flat in the apartment $A'_{ik}$ of the same dimension as $A'_{ik}$, and that $\partial Q_i^{\mathrm{fin}}$ lies in the interior of $\Gamma(c_i,c_k)$. The claim follows.
 
 Consequently, the extension of \( \Gamma(c_i,c_j)\cap\Gamma(c_i,c_k) \) in the apartments \(A'_{ij}\) and \(A'_{ik}\) must coincide with \(Q'_i\). Since we have that 
		$$ \partial Q_i^{\mathrm{fin}} \subset  \Gamma(c_i, c_j)\cap \Gamma(c_i, c_k)  \subset A'_{ij} \cap A'_{ik} =Q'_i,$$
		for any $C'_i \in U_i$, for $i \in \{1,2,3\}$, all the above implies that $\partial Q_i^{\mathrm{fin}} = (\partial Q'_i)^{\mathrm{fin}}$, for every $i \in \{1,2,3\}$.
		
		This means that for $r$ large enough we have the equality
		$$\zeta(C'_1,C'_2,C'_3) = \zeta(C_1,C_2,C_3),$$
		for any $C'_i \in U_i$ with $i \in \{1,2,3\}$. This proves the proposition. 
		
	\end{proof}
	
	\section{Configurations of non-generic triples of ideal chambers}
	\label{sec::configurations_non-generic}
	
	\begin{Def}
		Let $X$ be an affine building and let $A$ be an apartment of $X$. A bi-infinite geodesic line in $A$ is called a \textbf{bisector} if its ideal endpoints are the barycenters of two opposite ideal chambers in the ideal boundary $A^{\infty}$ of $A$.
	\end{Def}

	\begin{prop}\label{prop antipodal non generic case}
		Let $X$ be a thick affine building with associated affine Coxeter group $(W,S)$. If in the Coxeter complex of $(W,S)$ there is no root-flat orthogonal to bisectors, then every antipodal triple of ideal chambers of $X^\infty$ is actually in ideal-generic position, and in particular in affine-generic position. 
	\end{prop}
	
	\begin{proof}
		Let $\{C_1, C_2, C_3\}$ be an antipodal triple of ideal chambers in $X^{\infty}$.  For each distinct pair $i \neq j \in \{1,2,3\}$, let $A_{ij}$ be the unique affine apartment in $X$ whose ideal boundary contains both $C_i$ and $C_j$.  Let $\xi_i$ be the barycenter of the ideal chamber $C_i$, for $i \in \{1,2,3\}$.
		
		Suppose that the antipodal triple $\{C_1, C_2, C_3\}$ is not in ideal-generic position. Then, by definition, the intersection $A_{12}^{\infty} \cap A_{23}^{\infty} \cap A_{13}^{\infty}$ is non-empty. Thus, for every $i \in \{1,2,3\}$ we can find an unbounded,  geodesically convex subset  $V_{i} \subset \partial Q_i$ of a root-flat, where $Q_i := A_{ij} \cap A_{ik}$ (see the notation from Definition \ref{def::generic_position_1}), such that the ideal boundaries coincide:
		$$V_{i}^{\infty}= V_{j}^{\infty}, \quad \text{ for every } i \neq j \neq k \in \{1,2,3\}.$$ 
		
		Without loss of generality, we can assume that $V_i$ are of maximal dimension with the above properties. Moreover, notice the parallelism: 
		$$V_i  \text{ and } V_j \text{ are parallel in } A_{ij}, \quad  \text{ for every } i\neq j \in \{1,2,3\}.$$ 
		
		In addition, for each $i \in \{1,2,3\}$, fix some point $x_i $ in the interior of $V_i$. 	 
		
		For each pair $i \neq j$, consider the bi-infinite geodesic line $\gamma^{i}_{ij}$, respectively  $\gamma^{j}_{ij}$, in $A_{ij}$ passing through $x_i$, respectively $x_j$, with endpoints given by the barycenters of the ideal chambers $C_i$ and $C_j$. The geodesic line $\gamma^{i}_{ij}$ is  the bisector associated with $\{C_i, C_j\}$ through the point $x_i$. Notice that the bisectors $\gamma^{i}_{ij}, \gamma^{j}_{ij}$ are parallel in $A_{ij}$.
		
		We claim that, for each $i \in \{1,2,3\}$, the set $V_i$, viewed as a subset of the apartment $A_{ij}$ (more precisely, the root-flat of
		$A_{ij}$ spanned by $V_i$), is perpendicular to the bisector $\gamma^{i}_{ij}$ passing through the point $x_i \in V_i$. To prove this
		claim, it suffices to consider the case $i=1$, the apartment $A_{12}$, and the bisector $\gamma^{1}_{12}$.
		
		Indeed, choose any geodesic ray $r_1 := [x_1, \eta) \subset V_1$, with ideal endpoint $\eta \in V_1^{\infty}$. Let $\alpha$ denote the Euclidean angle at $x_1$ between the rays $r_1$ and $[x_1, \xi_1) \subset \gamma^{1}_{12}$.  Consider also the canonical retractions
		\[
		\rho_{A_{ij}, C_i}, \qquad \rho_{A_{ij}, C_j},
		\]
		from the building $X$ onto the apartment $A_{ij}$, based at the chambers $C_i$ and $C_j$, respectively, for each $i \neq j \in \{1,2,3\}$.
		
		To establish the desired perpendicularity, we ``rotate'' the ray $[x_1, \xi_1) \subset \gamma^{1}_{12}$ around the class of rays parallel to $r_1$ by successively applying the retractions $\rho_{A_{23}, C_2}$ and $\rho_{A_{12}, C_1}$.

		Since $C_2$ is opposite to both $C_1$ and $C_3$, we have
		\[
		\rho_{A_{23}, C_2}(C_1) = C_3, \quad  \rho_{A_{23}, C_2}(A_{12}) = A_{23}, \quad  \rho_{A_{23}, C_2}(V_1) \text{ parallel to } V_2, V_3.
		\]
		In particular, we obtain that  
		$$\rho_{A_{23}, C_2}([x_1, \xi_1)) \text{ is parallel to } [x_3, \xi_3), \quad \rho_{A_{23}, C_2}(\xi_1)=\xi_3, \quad  \rho_{A_{23}, C_2}(\eta)=\eta.$$
		
		By parallelism, the Euclidean angle between the rays $\rho_{A_{23}, C_2}(r_1)$ and $\rho_{A_{23}, C_2}([x_1, \xi_1))$ remains equal to $\alpha$, which is the same as the Euclidean angle at $x_3$ between the rays $r_3: =  [x_3, \eta)$ and $[x_3, \xi_3) \subset \gamma^{3}_{23}$.

		\medskip
		Similarly, since $C_1$ is opposite to both $C_2$ and $C_3$, we obtain
		\[
		\rho_{A_{12}, C_1}(C_3) = C_2, \quad \rho_{A_{12}, C_1}(A_{13}) = A_{12}, \quad  \rho_{A_{12}, C_1}(V_3) \text{ parallel to } V_1, V_2.
		\]
		In particular,
		$$\rho_{A_{12}, C_1}([x_3, \xi_3)) \text{ is parallel to } [x_2, \xi_2), \quad \rho_{A_{12}, C_1}(\xi_3)=\xi_2, \quad  \rho_{A_{12}, C_1}(\eta)=\eta.$$
		Again by parallelism, the Euclidean angle between the rays $\rho_{A_{12}, C_1}(r_3)$ and $\rho_{A_{12}, C_1}([x_3, \xi_3))$ remains equal to $\alpha$. And this angle is the same as the Euclidean angle at $x_2$ between the rays $r_2: =  [x_2, \eta)$ and $[x_2, \xi_2) \subset \gamma^{2}_{12}$. 
		
		Notice that a retraction $\rho_{A,C}$ on $X$ centered at an ideal chamber \( C  \in \Ch(X^\infty)\) onto an affine apartment $A$ of $X$ whose ideal boundary also contains \( C\) is distance-preserving when retracting any affine apartment $A'$ of \( X \) whose ideal boundary also contains \( D \). In particular, the angles considered above are preserved under such a retraction $\rho_{A,C}$.
		
		\medskip
		Since we are working with the parallel rays $r_1, r_2, r_3$, the angle $\alpha$ at $x_1$ between the rays
		$r_1$ and $[x_1, \xi_1) \subset \gamma^{1}_{12}$ can coincide with the  corresponding angle at $x_2$ between the rays $r_2$ and $[x_2, \xi_2) \subset
		\gamma^{2}_{12}$ \emph{if and only if} $\alpha = 90^\circ$.  This completes the proof of the claim and yields a contradiction to the
		hypothesis that no root-flats exist in $A_{12}$ which are perpendicular to bisectors. The proposition is therefore established.
	\end{proof}
	
	By Proposition~\ref{prop antipodal non generic case}, in order to construct geometric configurations of non--ideal-generic and non--affine-generic triples of ideal chambers in \(X^{\infty}\), it suffices to consider affine buildings whose associated affine Coxeter groups admit root-flats that are perpendicular to bisectors.
	
	It is straightforward to verify that if the affine Coxeter group is of type \(\widetilde{C}_2\), or \(\widetilde{G}_2\), then no such root-flats exist. In these rank--two cases, the only root-flats are affine walls, and none of them are perpendicular to bisectors. 
	
	In contrast, if the Coxeter group is of type \(\widetilde{A}_2\), then there do exist affine walls that are perpendicular to bisectors. This is precisely the setting in which non-generic geometric configurations may occur.

	\subsection{A classification of finite Weyl groups by the criterion in Proposition~\ref{prop antipodal non generic case}}

	Throughout this section, $\Phi \subset V = \mathbb{R}^n$ is an irreducible finite root system with finite Weyl group $ W_{\mathrm{fin}}$, $S_{\mathrm{fin}}$ is the set of simple roots, and $w_0$ is the longest element of $ W_{\mathrm{fin}}$. The Weyl vector is defined by
	\[
	\rho := \tfrac12 \sum_{\alpha \in \Phi^+} \alpha,
	\]
	which corresponds to the direction of the bisector of the Weyl cone of the corresponding affine Weyl group $(W,S)$ associated with $( W_{\mathrm{fin}}, S_{\mathrm{fin}})$.  A \emph{proper root subspace} of $V$ is a subspace of the form $\mathrm{Span}(T)$ for some $T \subsetneq \Phi$ with $\mathrm{Span}(T) \subsetneq V$.
	
	\begin{prop}\label{prop:geom_criterion}
		The following conditions are equivalent: 
		\begin{enumerate}
			\item There is no root-flat in the Coxeter complex of the affine Weyl group $(W,S)$ that is perpendicular to a bisector; 
			\item For every root subspace \(U=\Span(T)\subsetneq V\), one has \(\rho\notin U\).
		\end{enumerate}
		
	\end{prop}
	
	\begin{proof} For a root $\alpha \in \Phi$, we denote by $H_\alpha$ the corresponding wall (hyperplane) in $V$ through the origin and perpendicular to $\alpha$.  
		
		Note that a bisector in the Coxeter complex of the affine Weyl group $(W,S)$ has direction $w\rho$ for some $w \in W_{\mathrm{fin}}$, and a root-flat in the Coxeter complex of $(W,S)$ through the origin is of the form
		\[
		\bigcap_{\alpha \in T} H_\alpha = \mathrm{Span}(T)^\perp
		\]
		for some $T \subset \Phi$.  
		
		The perpendicularity condition reads $\mathrm{Span}(T)^\perp \subset (w\rho)^\perp$, equivalently $w\rho \in \mathrm{Span}(T)$. Applying $w^{-1}$, this becomes $\rho \in \mathrm{Span}(w^{-1}T) = \mathrm{Span}(T')$ for a (still proper) subset $T' \subsetneq \Phi$. Conversely, if $\rho \in \mathrm{Span}(T')$ for some proper $T' \subsetneq \Phi$, then taking $w=e$ and $T=T'$ exhibits a root-flat $\mathrm{Span}(T)^\perp$ that is perpendicular to the bisector through the origin in direction $\rho$. This proves the equivalence.
	\end{proof}
	
	Notice that checking the perpendicularity of a bisector on root-flats reduces to checking it on the one-dimensional root-flats, which are the \( W_{\mathrm{fin}} \)-orbits of the root-flats \( \bigcap_{\alpha \in S_{\mathrm{fin}}  \setminus \{\alpha_i\}} H_\alpha \), for each simple root \( \alpha_i \in S_{\mathrm{fin}}  \). The next lemma translates this idea into the language of spanning the Weyl vector by fewer roots of \( \Phi \).
	
	\begin{lem} \label{lem:reduction}
		The Weyl vector \( \rho \) lies in a proper root subspace \( \mathrm{Span}(T) \) for some subset \( T \subsetneq \Phi \) if and only if there exists an element \( w \in W_{\mathrm{fin}}  \) and a simple root \( \alpha_i \in S_{\mathrm{fin}}  \) such that $w\rho \in \mathrm{Span}\bigl(\Phi_{S_{\mathrm{fin}}  \setminus \{\alpha_i\}}\bigr)$.
	\end{lem}
	\begin{proof}
		Suppose first that \(\rho\in U:=\Span(T)\) for some proper root subspace \(U\subsetneq V\).  Put
		\[
		L:=U^\perp=\bigcap_{\alpha\in T}H_\alpha, 
		\]
		which is an intersection of affine walls that is not reduced to a point. 

Let $\Psi:=\Phi\cap U$. For $\alpha\in\Phi$, denote by \[ s_\alpha(x)=x-\frac{2(\alpha,x)}{(\alpha,\alpha)}\alpha \] the reflection in $V$ associated with the root $\alpha$. If $\alpha,\beta\in\Psi$, then $\alpha,\beta\in U$. Then $s_\alpha$ preserves $U$ and we have \(s_\alpha(\beta)\in U\). As $\Phi$ is a root system, $s_\alpha(\beta)\in\Phi$. Thus \[ s_\alpha(\beta)\in\Phi\cap U=\Psi, \] showing that $\Psi$ is a root subsystem. 

Recall that parabolic subgroups of $W_{\mathrm{fin}}$ correspond to root subsystems of $\Phi$ of lower rank. Every such subsystem of \( \Phi \) is \( W_{\mathrm{fin}} \)-conjugate to a standard parabolic subsystem \( \Phi_{S'} \) for some subset \( S' \subset S_{\mathrm{fin}}  \). This implies that $L  $ is \(W_{\mathrm{fin}}\)-conjugate to a standard parabolic flat. Thus, there are \(w\in W_{\mathrm{fin}}\) and a subset \(J\subsetneq S_{\mathrm{fin}}\) such that
		\[
		wL=\bigcap_{\alpha\in J}H_\alpha .
		\]
		Taking orthogonal complements gives $ wU=\Span(J)$ and since \(J\subsetneq S_{\mathrm{fin}}\), we can choose \(\alpha_i\in S_{\mathrm{fin}}\setminus J\).  Then
		\[
		w\rho\in wU=\Span(J)
		\subseteq \Span\bigl(\Phi_{S_{\mathrm{fin}}\setminus\{\alpha_i\}}\bigr),
		\]
		which is the desired condition.
		
		Conversely, if \(w\rho\in\Span(\Phi_{S_{\mathrm{fin}}\setminus\{\alpha_i\}})\), then applying \(w^{-1}\) gives
		\[
		\rho\in \Span\bigl(w^{-1}\Phi_{S_{\mathrm{fin}}\setminus\{\alpha_i\}}\bigr),
		\]
		and the right-hand side is a proper root subspace of \(V\).  This proves the equivalence.
	\end{proof}
	
	Let us apply Lemma \ref{lem:reduction} to the list of finite Coxeter groups \( W_{\mathrm{fin}} \). Thus, it suffices to check whether there exists an element in the orbit \( W_{\mathrm{fin}} \cdot \rho \) that lies in one of the hyperplanes \( \mathrm{Span}\bigl(\Phi_{S_{\mathrm{fin}} \setminus \{\alpha_i\}}\bigr) \).	  
	
	For an overview of finite Coxeter groups and their associated root systems, one may consult \cite[pp.~22--24]{SamMetaplecticNotes}.

	\subsubsection{Coxeter types where root-flats are not perpendicular to bisectors}
	\label{subsec::not_perp_bisectors}
	
	In the following case-by-case calculations we use Bourbaki's standard numbering and realizations of the irreducible finite root systems. We follow specifically \cite[Ch.~IV, \S1]{bourbaki_IV-VI_02}. The explicit root-system data, including the positive roots, highest roots and fundamental weights, are taken from \cite[Ch.~VI, \S1, nos.~6, 8--10,
	and Planches~I--IX]{bourbaki_IV-VI_02}.  Since
	\[
	\rho=\frac12\sum_{\alpha\in\Phi^+}\alpha
	=\sum_{i=1}^{\operatorname{rk}\Phi}\omega_i,
	\]
	where the \(\omega_i\) are the fundamental weights, the Weyl vectors are obtained directly from these tables as well. 
	
	\begin{prop}
		Let \((W_{\mathrm{fin}}, S_{\mathrm{fin}})\) be a finite Coxeter system of type \( B_2 = C_2 \), \( G_2 \), \( B_3 \). Then \( W_{\mathrm{fin}} \) satisfies the conditions of Proposition \ref{prop:geom_criterion}; that is, there are no root-flats that are perpendicular to bisectors.
	\end{prop}
	
	\begin{proof}
		Assume that \((W_{\mathrm{fin}}, S_{\mathrm{fin}})\)  is of type $B_2=C_2$. With the standard realization $\Phi=\{\pm e_1,\pm e_2,\pm e_1\pm e_2\}$ and positive roots $\Phi_+=\{e_1,e_2, e_1\pm e_2\}$, we have $\rho=\tfrac32 e_1 +\tfrac12 e_2$.  Proper root subspaces are the lines $\bbR e_1,\bbR e_2,\bbR(e_1+e_2),\bbR(e_1-e_2)$, none of which contains $\rho$.
		
		Assume now that \((W_{\mathrm{fin}}, S_{\mathrm{fin}})\) is of type $G_2$. Let $\alpha,\beta$ be the simple roots, with $\alpha$ short and $\beta$ the long root. 
		The six positive roots are 
		\[\Phi^+ = \{\alpha,\beta,\alpha+\beta,2\alpha+\beta,3\alpha+\beta,3\alpha+2\beta\},\]
		so that $\rho=5\alpha+3\beta$.  Proper root subspaces of $V=\bbR^2$ are
		one-dimensional and are the lines $\bbR r$ for each positive root $r$.
		The vector $5\alpha+3\beta$ is not a scalar multiple of any of these:
		the ratios $\alpha:\beta$ of $r$ for $r$ running over the positive
		roots are $\{1:0,\;0:1,\;1:1,\;2:1,\;3:1,\;3:2\}$, none equal to $5:3$.
		
		Assume finally that \((W_{\mathrm{fin}}, S_{\mathrm{fin}})\) is of type $B_3$. Here $\rho=(\tfrac52,\tfrac32,\tfrac12)$.  By Lemma~\ref{lem:reduction},
		it suffices to check the three maximal standard parabolic spans
		\begin{align}
			& \Span(\Phi_{S\setminus\{\alpha_1\}}) = \{x_1=0\}, \nonumber \\
			& \Span(\Phi_{S\setminus\{\alpha_2\}}) = \{x_1+x_2=0\}, \nonumber \\
			& \Span(\Phi_{S\setminus\{\alpha_3\}}) = \{x_1+x_2+x_3=0\}, \nonumber
		\end{align}
		
		together with their $W_{\mathrm{fin}}$-orbits under signed permutations.
		\begin{itemize}
			\item $W_{\mathrm{fin}}$-orbit of $\{x_1=0\}$: hyperplanes $\{x_i=0\}$.  All three
			coordinates of $\rho$ are nonzero, hence $\rho$ is in none of them.
			\item $W_{\mathrm{fin}}$-orbit of $\{x_1+x_2=0\}$: hyperplanes $\{x_i\pm x_j=0\}$ for
			$i\ne j$.  Containment requires
			$\pm\tfrac52\pm\tfrac32=0$, $\pm\tfrac52\pm\tfrac12=0$ or
			$\pm\tfrac32\pm\tfrac12=0$, none of which holds.
			\item $W_{\mathrm{fin}}$-orbit of $\{x_1+x_2+x_3=0\}$: hyperplanes
			$\{\epsilon_1 x_1+\epsilon_2 x_2+\epsilon_3 x_3 = 0\}$, $\epsilon_i\in\{\pm1\}$.
			Containment requires $\pm 5\pm 3\pm 1=0$; the possible values are
			$\pm9,\pm7,\pm3,\pm1$, none zero.
		\end{itemize}
	\end{proof}
	
	\subsubsection{Coxeter types where root-flats are perpendicular to bisectors}
	\label{subsec::perp_bisectors}
	
	\begin{prop}
		\label{prop:An}
		For every $n\ge 2$, the root system $A_n$ violates the conditions of Proposition~\ref{prop:geom_criterion}, that is, there are root-flats that are perpendicular to bisectors.
	\end{prop}
	
	\begin{proof}
		
		We work in  $V = \left\{ x \in \mathbb{R}^{n+1} \;\middle|\; \sum_{i=1}^{n+1} x_i = 0 \right\}$
		with the standard basis $(e_1,\dots,e_{n+1})$. The root system of type $A_n$ is 
		\[
		\Phi(A_n)
		=
		\{ \pm (e_i - e_j) \mid 1 \le i < j \le n+1 \}.
		\]
		
		The standard choice of positive roots is 
		\[
		\Phi^+(A_n)
		=
		\{ e_i - e_j \mid 1 \le i < j \le n+1 \}.
		\]
		
		The standard choice of the simple roots is
		\[
		\alpha_i = e_i - e_{i+1}, \quad i=1,\dots,n.
		\]
		
		The Weyl group of type $A_n$ is the group of permutations $
		W(A_n) \cong S_{n+1}$.
		
		For $k=1,\dots,n$, the lines corresponding to maximal parabolic subgroups, thus removing $\alpha_k$, is
		\[
		L_k = \mathbb{R} v_k,
		\]
		where
		\[
		v_k =
		(\underbrace{n+1-k,\dots,n+1-k}_{k},\;
		\underbrace{-k,\dots,-k}_{n+1-k}).
		\]
		
		For $k=1,\dots,n$, the hyperplane perpendicular to $L_k$ is
		\[
		H_k
		=
		\{ x \in V \mid \langle x, v_k \rangle = 0 \}.
		\]
		A direct computation gives
		\[
		\langle x, v_k \rangle
		= (n+1-k)(x_1 + \cdots + x_k)
		- k(x_{k+1} + \cdots + x_{n+1}).
		\]
		Using the relation $\sum_{i=1}^{n+1} x_i = 0$, this simplifies to $\langle x, v_k \rangle = (n+1)(x_1 + \cdots + x_k)$.
		Hence,
		\[
		H_k
		=
		\left\{ x \in V \;\middle|\; x_1 + \cdots + x_k = 0 \right\}.
		\]
		
		The $W(A_n)$-orbit of $H_k$ is
		\[
		\mathcal{H}_k
		=
		\left\{
		x \in V \;\middle|\;
		\sum_{i \in I} x_i = 0
		\;\middle|\;
		I \subset \{1,\dots,n+1\},\ |I| = k
		\right\}.
		\]
		Thus the $W(A_n)$-orbits are determined by the size $k$, i.e.\ subsets of $\{1,\dots,n+1\}$.
		
		Notice that the Weyl vector is $\rho_{A_n} = \frac12 \sum_{\alpha \in \Phi^+} \alpha
		= \sum_{i=1}^{n+1} \frac{n+2-2i}{2}\, e_i$. We determine for which hyperplanes $H_I = \left\{ x \in V \;\middle|\; \sum_{i \in I} x_i = 0 \right\}$ the Weyl vector $\rho_{A_n} = \tfrac{1}{2}(n, n-2, \dots, -n)$ lies.
		
		A direct computation shows that
		\[
		\sum_{i \in I} \rho_i = \sum_{i \in I} \frac{n+2-2i}{2}
		= \frac{|I|(n+2)}{2} - \sum_{i \in I} i.
		\]
		Hence,
		\[
		\rho_{A_n} \in H_I
		\quad \Longleftrightarrow \quad
		\sum_{i \in I} i = \frac{|I|(n+2)}{2}.
		\]
		
		Equivalently, this means that the average of the indices in $I$
		is equal to $\frac{n+2}{2}$. More precisely,  for any $n$ and any $i$, the subset $I = \{i,\, n+2-i\}$
		satisfies $\rho_{A_n} \in H_I$.
	\end{proof}
	
	\begin{prop}
		\label{prop:Bn}
		For every $n\ge 4$, the root system $B_n$ violates the conditions of Proposition~\ref{prop:geom_criterion}, that is, there are root-flats that are perpendicular to bisectors.
	\end{prop}
	
	\begin{proof}
		We work in $V = \mathbb{R}^n$ with the standard orthonormal basis $(e_1,\dots,e_n)$. The root system of type $B_n$ is
		\[
		\Phi(B_n)
		=
		\{ \pm e_i \mid 1 \le i \le n \}
		\;\cup\;
		\{ \pm e_i \pm e_j \mid 1 \le i < j \le n \}.
		\]
		
		A standard choice of positive roots is
		\[
		\Phi^+(B_n)
		=
		\{ e_i \mid 1 \le i \le n \}
		\;\cup\;
		\{ e_i \pm e_j \mid 1 \le i < j \le n \}.
		\]
		
		A standard choice of the simple roots is given by
		\[
		\alpha_1 = e_1 - e_2,\quad
		\alpha_2 = e_2 - e_3,\quad \dots,\quad
		\alpha_{n-1} = e_{n-1} - e_n,\quad
		\alpha_n = e_n.
		\]
		
		The Weyl group of type $B_n$ is the group of signed permutations $
		W(B_n) \cong (\mathbb{Z}_2)^n \rtimes S_n$.  The lines corresponding to maximal parabolic subgroups are
		\[
		L_k = \mathbb{R}v_k, \qquad
		v_k = (\underbrace{1,\dots,1}_{k},0,\dots,0),
		\quad k=1,\dots,n.
		\]
		
		The hyperplane perpendicular to $L_k$ is
		\[
		\Span(\Phi_{S\setminus\{\alpha_k\}}) = H_k = v_k^\perp
		=
		\left\{ x \in \mathbb{R}^n \;\middle|\; x_1 + \cdots + x_k = 0 \right\}.
		\]
		
		Under the action of $W(B_n)$, the orbit of $L_k$ consists of all lines
		\[
		\mathcal{O}_k
		=
		\left\{
		\mathbb{R}(\pm e_{i_1} \pm \cdots \pm e_{i_k})
		\;\middle|\;
		1 \le i_1 < \cdots < i_k \le n
		\right\}.
		\]
		
		Similarly, the orbit of $H_k$ consists of all hyperplanes
		\[
		\left\{
		x \in \mathbb{R}^n \;\middle|\;
		\sum_{i \in I} \varepsilon_i x_i = 0
		\right\},
		\quad |I| = k,\ \varepsilon_i \in \{\pm 1\}.
		\]
		
		Notice that the Weyl vector is $\rho_{B_n}=\bigl(\tfrac{2n-1}{2},\tfrac{2n-3}{2},\dots,\tfrac12\bigr)$. Then by the above description of the $W(B_n)$-orbits of $\Span(\Phi_{S\setminus\{\alpha_k\}})$, and since $1-3-5+7=0$, we have
		\[
		\rho_{B_n}\cdot(e_1-e_2-e_3+e_4)
		= \tfrac{(2n-1)-(2n-3)-(2n-5)+(2n-7)}{2}=0,
		\]
		so $\rho_{B_n}\in (e_1-e_2-e_3+e_4)^\perp$, which is the Weyl-conjugate $\{x_1-x_2-x_3+x_4=0\}$ of the hyperplane $H_4=\Span(\Phi_{S\setminus\{\alpha_4\}})$. This is a proper root subspace, so $\rho_{B_n}$ lies in a proper root subspace.
	\end{proof}
	
	\begin{prop}
		\label{prop:Cn} 
		For every $n\ge 3$, the root system $C_n$ violates the conditions of Proposition~\ref{prop:geom_criterion}, that is, there are root-flats that are perpendicular to bisectors.
	\end{prop}
	
	\begin{proof}
		We work in $V = \mathbb{R}^n$ with basis $(e_1,\dots,e_n)$. The root system of type $C_n$ is
		\[\Phi(C_n)
		=
		\{ \pm 2e_i \mid 1 \le i \le n \}
		\;\cup\;
		\{ \pm e_i \pm e_j \mid 1 \le i < j \le n \}.
		\]
		
		A standard choice of positive roots is
		\[
		\Phi^+(C_n)
		=
		\{ 2e_i \mid 1 \le i \le n \}
		\;\cup\;
		\{ e_i \pm e_j \mid 1 \le i < j \le n \}.
		\]
		
		A standard choice of the simple roots is given by
		\[
		\alpha_1 = e_1 - e_2,\quad \dots,\quad
		\alpha_{n-1} = e_{n-1} - e_n,\quad
		\alpha_n = 2e_n.
		\]
		
		The Weyl group of type $C_n$ is the group of signed permutations $ W(C_n) = W(B_n) \cong (\mathbb{Z}_2)^n \rtimes S_n$.  The lines corresponding to maximal parabolic subgroups are the same as in the case of $B_n$:
		\[
		L_k = \mathbb{R}(\underbrace{1,\dots,1}_{k},0,\dots,0), \quad k=1,\dots,n.
		\]
		
		The hyperplane perpendicular to $L_k$ is the same as in the case of $B_n$
		\[
		\Span(\Phi_{S\setminus\{\alpha_k\}}) = H_k = v_k^\perp
		=
		\left\{ x \in \mathbb{R}^n \;\middle|\; x_1 + \cdots + x_k = 0 \right\}.
		\]
		The $W(C_n)$-orbits of those lines and hyperplanes are identical to the case of $B_n$.
		
		Notice that the Weyl vector is $\rho_{C_n}=(n,n-1,\dots,1)$. Then by the above description of the $W(C_n)$-orbits of $\Span(\Phi_{S\setminus\{\alpha_k\}})$:
		\begin{itemize}
			\item For $n\ge 4$, the identity $-n+(n-1)+(n-2)-(n-3)=0$ gives
			\[
			\rho_{C_n}\cdot(-e_1+e_2+e_3-e_4)=0,
			\]
			so $\rho_{C_n}\in(-e_1+e_2+e_3-e_4)^\perp=\{-x_1+x_2+x_3-x_4=0\}$, a Weyl-conjugate of $H_4=\Span(\Phi_{S\setminus\{\alpha_4\}})$, hence a proper root subspace.
			\item For $n=3$, the identity $3-2-1=0$ gives
			\[
			\rho_{C_3}\cdot(e_1-e_2-e_3)=0,
			\]
			so $\rho_{C_3}\in(e_1-e_2-e_3)^\perp=\{x_1-x_2-x_3=0\}$, a Weyl-conjugate of $H_3=\Span(\Phi_{S\setminus\{\alpha_3\}})$, hence a proper root subspace.
		\end{itemize}
	\end{proof}
	
	\begin{prop}\label{prop:Dn} 
		For every $n\ge 3$, the root system $D_n$ violates the conditions of Proposition~\ref{prop:geom_criterion}, that is, there are root-flats that are perpendicular to bisectors.
	\end{prop}
	
	\begin{proof}
		We work in $V = \mathbb{R}^n$ with basis $(e_1,\dots,e_n)$. The root system of type $D_n$ is
		\[
		\Phi(D_n)
		=
		\{ \pm e_i \pm e_j \mid 1 \le i < j \le n \}.
		\]
		
		A standard choice of positive roots is
		\[
		\Phi^+(D_n)
		=
		\{ e_i \pm e_j \mid 1 \le i < j \le n \}.
		\]
		
		A standard choice of the simple roots is given by
		\[
		\alpha_1 = e_1 - e_2,\quad \dots,\quad
		\alpha_{n-1} = e_{n-1} - e_n,\quad
		\alpha_n = e_{n-1} + e_n.
		\]
		
		The Weyl group of type $D_n$ is the group of even signed permutations $ W(D_n) =
		\left\{
		(\varepsilon_i,\sigma) \in (\mathbb{Z}_2)^n \rtimes S_n
		\;\middle|\;
		\prod_{i=1}^n \varepsilon_i = 1
		\right\},
		$.  
		
		The lines corresponding to maximal parabolic subgroups are
		\[
		L_k =
		\begin{cases}
			\mathbb{R}(\underbrace{1,\dots,1}_{k},0,\dots,0),
			& 1 \le k \le n-2, \\[6pt]
			\mathbb{R}(1,\dots,1,-1),
			& k = n-1, \\[6pt]
			\mathbb{R}(1,\dots,1,1),
			& k = n.
		\end{cases}
		\]
		
		The hyperplanes perpendicular to the above lines are
		\[
		H_k =
		\begin{cases}
			\left\{ x \in \mathbb{R}^n \;\middle|\; x_1 + \cdots + x_k = 0 \right\},
			& 1 \le k \le n-2, \\[6pt]
			\left\{ x \in \mathbb{R}^n \;\middle|\; x_1 + \cdots + x_{n-1} - x_n = 0 \right\},
			& k = n-1, \\[6pt]
			\left\{ x \in \mathbb{R}^n \;\middle|\; x_1 + \cdots + x_n = 0 \right\},
			& k = n.
		\end{cases}
		\]
		
		Notice that under the action of $W(D_n)$, we can:
		\begin{itemize}
			\item permute indices, sending any subset $I \subset \{1,\dots,n\}$ with $|I|=k$ to another such subset;
			\item flip signs, introducing factors $\varepsilon_i \in \{\pm1\}$, subject to the constraint that the total number of sign changes is even.
		\end{itemize}
		
		The $W(D_n)$-orbits of those hyperplanes are as follows.
		\begin{itemize}
			\item For $1 \le k \le n-2$, there is a single orbit consisting of all hyperplanes
			\[
			\left\{
			x \in \mathbb{R}^n \;\middle|\;
			\sum_{i \in I} \varepsilon_i x_i = 0
			\right\},
			\quad |I| = k,\ \varepsilon_i = \pm1.
			\]
			
			\item For $k = n-1$, there is a single orbit consisting of all hyperplanes
			\[
			\left\{
			x \in \mathbb{R}^n \;\middle|\;
			\sum_{i \in I} \varepsilon_i x_i = 0
			\right\},
			\quad |I| = n-1,\ \varepsilon_i = \pm1.
			\]
			
			\item For $k = n$, the orbit splits into two distinct $W(D_n)$-orbits:
			\[
			\left\{
			x \in \mathbb{R}^n \;\middle|\;
			\sum_{i=1}^n \varepsilon_i x_i = 0,\;
			\prod_{i=1}^n \varepsilon_i = 1
			\right\}
			\]
			and
			\[
			\left\{
			x \in \mathbb{R}^n \;\middle|\;
			\sum_{i=1}^n \varepsilon_i x_i = 0,\;
			\prod_{i=1}^n \varepsilon_i = -1
			\right\}.
			\]
		\end{itemize}
		
		Notice that the Weyl vector is $\rho_{D_n} = (n-1, n-2, \dots, 1, 0)$.  By the above description of the $W(D_n)$-orbits of $\Span(\Phi_{S\setminus\{\alpha_k\}})$:
		\begin{itemize}
			\item For $n\ge 6$, the identity $(n-1)-(n-2)-(n-3)+(n-4)=0$ gives
			\[
			\rho_{D_n}\cdot(e_1-e_2-e_3+e_4)=0,
			\]
			so $\rho_{D_n}\in(e_1-e_2-e_3+e_4)^\perp$, a Weyl-conjugate of $H_4$, hence a proper root subspace.
			\item For $n=5$, the identity $(5-1)-(5-2)-(5-4)=4-3-1=0$ gives
			\[
			\rho_{D_5}\cdot(e_1-e_2-e_4)=0,
			\]
			so $\rho_{D_5}\in(e_1-e_2-e_4)^\perp$, a Weyl-conjugate of $H_3$, hence a proper root subspace.
			\item For $n=4$, the identity $(4-1)-(4-2)-(4-3)=3-2-1=0$ gives
			\[
			\rho_{D_4}\cdot(e_1-e_2-e_3)=0,
			\]
			so $\rho_{D_4}\in(e_1-e_2-e_3)^\perp$, a Weyl-conjugate of $H_3$, hence a proper root subspace.
		\end{itemize}
		
		For $n=3$, $\rho_{D_3} = (2,1,0)$, so $\rho_{D_3}\in\{x \in \mathbb{R}^3 \mid x_3 = 0\}$ directly. (In fact $D_3\cong A_3$, which already implies that there are root-flats perpendicular to bisectors by Prop.~\ref{prop:An}.)
	\end{proof}

	\begin{prop}
		\label{prop:F4}
		The root system $F_4$ violates the conditions of Proposition~\ref{prop:geom_criterion}, that is, there are root-flats that are perpendicular to bisectors.
	\end{prop}
	
	\begin{proof}
		In the standard realization, $\rho_{F_4}=(\tfrac{11}{2},\tfrac52,\tfrac32,\tfrac12)$.
		Reflecting in the short root $\alpha=\tfrac12(e_1-e_2-e_3-e_4)$,
		which has $\langle\alpha,\alpha\rangle=1$, gives
		$\langle\rho,\alpha\rangle=\tfrac12$ and hence
		\[
		s_\alpha(\rho) \;=\; \rho - \alpha \;=\; (5,3,2,1).
		\]
		Since $5-3-2=0$, $s_\alpha(\rho)\in\{x_1-x_2-x_3=0\}$.  This hyperplane
		is spanned by the $F_4$-roots $e_4,e_1+e_2,e_1+e_3$ (and contains
		$e_2-e_3$), hence is a proper root subspace.  Pulling back by $s_\alpha$
		shows $\rho_{F_4}$ itself lies in a proper root subspace.
	\end{proof}
	
	\begin{prop}
		\label{prop:E678}
		The root systems  \(E_6\), \(E_7\) and \(E_8\) violate the conditions of
		Proposition~\ref{prop:geom_criterion}, that is, there are root-flats that are perpendicular to bisectors.
	\end{prop}
	
	\begin{proof}
		We use Bourbaki's realization of \(E_6\) inside \(\mathbb R^8\), with standard basis \(e_1,\ldots,e_8\).  In this realization the root space is
		\[
		V=\Span_{\mathbb R}\Phi(E_6)
		=
		\{x\in\mathbb R^8 \mid x_6=x_7=-x_8\},
		\]
		so \(\dim V=6\).  The Weyl vector is
		\[
		\rho_{E_6}=(0,1,2,3,4,-4,-4,4).
		\]
		Consider the following five roots of \(E_6\):
		\[
		\begin{aligned}
			\beta_1&=e_1+e_2,\\
			\beta_2&=e_1+e_3,\\
			\beta_3&=e_1+e_4,\\
			\beta_4&=e_2+e_3,\\
			\beta_5&=\frac12(e_1+e_2+e_3+e_4+e_5-e_6-e_7+e_8).
		\end{aligned}
		\]
		The first four roots belong to the \(D_5\)-subsystem \(\{\pm e_i\pm e_j\mid 1\le i<j\le5\}\subset\Phi(E_6)\), and \(\beta_5\) is one of the half-roots in Bourbaki's model.  Let
		\[
		L:=\Span_{\mathbb R}\{\beta_1,\beta_2,\beta_3,\beta_4,\beta_5\}.
		\]
		These five roots are linearly independent.  Indeed, in a relation
		\[
		a_1\beta_1+a_2\beta_2+a_3\beta_3+a_4\beta_4+a_5\beta_5=0,
		\]
		the \(e_6,e_7,e_8\)-coordinates imply \(a_5=0\).  The \(e_4\)-coordinate then gives \(a_3=0\), and the remaining \(e_1,e_2,e_3\)-coordinates give
		\[
		a_1+a_2=0,
		\qquad
		a_1+a_4=0,
		\qquad
		a_2+a_4=0,
		\]
		hence \(a_1=a_2=a_4=0\).  Thus \(\dim L=5\), so \(L\) is a proper root subspace of the six-dimensional root space \(V\).
		
		Finally, a direct calculation gives
		\[
		\rho_{E_6}
		=
		-2\beta_1-\beta_2-\beta_3-\beta_4+8\beta_5.
		\]
		Therefore \(\rho_{E_6}\in L\), and Proposition~\ref{prop:geom_criterion} gives a root-flat perpendicular to a bisector.

		We turn to $E_8$. In this realization the root system is given by
		\[
		\{\pm e_i\pm e_j\mid 1\le i<j\le 8\}
		\cup
		\left\{
		\frac12\sum_{i=1}^8 \varepsilon_i e_i
		\;\middle|\;
		\varepsilon_i\in\{\pm1\},\;
		\#\{i\mid \varepsilon_i=-1\}\equiv 0 \pmod 2
		\right\}.
		\]
		The simple roots are numbered as in Bourbaki:
		\[
		\begin{aligned}
			\alpha_1
			&=
			\frac12(e_1+e_8-e_2-e_3-e_4-e_5-e_6-e_7),\\
			\alpha_2&=e_1+e_2,\\
			\alpha_3&=e_2-e_1,\\
			\alpha_4&=e_3-e_2,\\
			\alpha_5&=e_4-e_3,\\
			\alpha_6&=e_5-e_4,\\
			\alpha_7&=e_6-e_5,\\
			\alpha_8&=e_7-e_6.
		\end{aligned}
		\]
		Since \(E_8\) is simply laced, the Weyl vector is characterized by
		\[
		\langle \rho_{E_8},\alpha_i\rangle=1,
		\qquad i=1,\ldots,8.
		\]
		Summing the fundamental weights gives
		\[
		\rho_{E_8}=(0,1,2,3,4,5,6,23).
		\]
		
		Set
		\[
		U_8:=\{x\in\mathbb R^8\mid x_1=0\}.
		\]
		The roots
		\[
		\pm e_i\pm e_j,
		\qquad 2\le i<j\le 8,
		\]
		belong to \(\Phi(E_8)\) and form a root subsystem of type \(D_7\). This
		subsystem spans \(U_8\). For instance, the seven roots
		\[
		e_2-e_3,\ e_3-e_4,\ e_4-e_5,\ e_5-e_6,\ e_6-e_7,\ e_7-e_8,\ e_7+e_8
		\]
		are linearly independent and all lie in \(U_8\). Hence \(U_8\) is a proper
		root subspace of the \(E_8\)-root space. Since the first coordinate of
		\(\rho_{E_8}\) is zero, we have
		\[
		\rho_{E_8}\in U_8.
		\]
		Thus \(\rho_{E_8}\) lies in a proper root subspace. By
		Proposition~\ref{prop:geom_criterion}, the affine Coxeter complex of type
		\(\widetilde E_8\) admits root-flats perpendicular to bisectors.

		We now treat \(E_7\). In Bourbaki's realization, \(E_7\) is the root
		subsystem of \(E_8\) generated by \(\alpha_1,\ldots,\alpha_7\). Its root
		space is
		\[
		V_7
		=
		\{x\in\mathbb R^8\mid x_7+x_8=0\},
		\]
		and
		\[
		\Phi(E_7)=\Phi(E_8)\cap V_7.
		\]
		Again, since \(E_7\) is simply laced, its Weyl vector is characterized by
		\[
		\langle \rho_{E_7},\alpha_i\rangle=1,
		\qquad i=1,\ldots,7.
		\]
		Using the above Bourbaki simple roots, one obtains
		\[
		\rho_{E_7}
		=
		\left(0,1,2,3,4,5,-\frac{17}{2},\frac{17}{2}\right).
		\]
		
		Consider the subspace
		\[
		U_7
		:=
		V_7\cap\{x_1=0\}
		=
		\{x\in\mathbb R^8\mid x_1=0,\ x_7+x_8=0\}.
		\]
		This is a proper subspace of \(V_7\), of dimension \(6\). We claim that it is
		a root subspace. Indeed, the following six roots belong to \(\Phi(E_7)\):
		\[
		\begin{aligned}
			\beta_1&=e_2-e_3,\\
			\beta_2&=e_3-e_4,\\
			\beta_3&=e_4-e_5,\\
			\beta_4&=e_5-e_6,\\
			\beta_5&=e_5+e_6,\\
			\beta_6&=e_7-e_8.
		\end{aligned}
		\]
		They all lie in \(U_7\). Moreover, they are linearly independent. Indeed, a
		linear relation
		\[
		a_1\beta_1+\cdots+a_6\beta_6=0
		\]
		has \(e_7,e_8\)-coordinates \(a_6\) and \(-a_6\), so \(a_6=0\). The remaining
		relation is supported on \(e_2,\ldots,e_6\):
		\[
		\begin{aligned}
			0
			&=
			a_1(e_2-e_3)
			+a_2(e_3-e_4)
			+a_3(e_4-e_5)  \\
			&\qquad
			+a_4(e_5-e_6)
			+a_5(e_5+e_6).
		\end{aligned}
		\]
		Looking successively at the \(e_2,e_3,e_4,e_5,e_6\)-coordinates gives
		\[
		a_1=0,\qquad
		-a_1+a_2=0,\qquad
		-a_2+a_3=0,\qquad
		-a_3+a_4+a_5=0,\qquad
		-a_4+a_5=0.
		\]
		Thus \(a_1=a_2=a_3=a_4=a_5=0\). Hence the six roots
		\(\beta_1,\ldots,\beta_6\) are linearly independent. Since
		\(\dim U_7=6\), they span \(U_7\). Therefore \(U_7\) is a proper root subspace
		of \(V_7\).
		
		Finally, the first coordinate of \(\rho_{E_7}\) is zero and $-\frac{17}{2}+\frac{17}{2}=0$ so $\rho_{E_7}\in U_7$. Thus \(\rho_{E_7}\) lies in a proper root subspace. By
		Proposition~\ref{prop:geom_criterion}, the affine Coxeter complex of type
		\(\widetilde E_7\) admits root-flats perpendicular to bisectors.
		
	\end{proof}

	\subsection{Configurations of non--affine-generic triples of ideal chambers}
	\label{subsec::conf_non_affine_generic}
	
	Assume that the locally finite affine building \(X\) is of an affine Coxeter 
	type that admits root-flats perpendicular to bisectors. Let \(F\) be a root-flat in \(X\) that is perpendicular to some bisector ($F$ not necessarily of maximal dimension with this property).  Choose two opposite ideal simplices \(\eta_{+}, \eta_{-} \in X^{\infty}\) 
	contained in \(F^{\infty}\) and of maximal dimension in \(F^{\infty}\).  
	As recalled in Section~\ref{sec::sub_buildings_affine_flats}, the space
	\[
	X(\eta_{+}, \eta_{-}) \;\cong\; \mathbb{R}^{\,n - |I|} \times X_I,
	\]
	where \(n = \dim X\), \(|I| = \dim X_I\), and the flat \(F\) is isometric to 
	\(\mathbb{R}^{\,n - |I|}\).
	
	\medskip
	
	We write \((W_I, I)\) for the finite Coxeter group associated with the 
	spherical building \(X_I^\infty\); note that this group embeds naturally in the 
	finite Weyl group \((W_{\mathrm{fin}}, S)\) of \(X\).  
	We also denote by \((W^{\mathrm{aff}}_I, I_{\mathrm{aff}})\) the affine Coxeter 
	group of the affine building \(X_I\).
	
	\medskip
	
	Although \(F\) is perpendicular to the bisectors of \(X^{\infty}\) (even when \(F\) is of maximal dimension with this property), this does not imply that \(X_{I}\) contains no root-flat perpendicular to the bisectors of \(X_{I}^{\infty}\).  Indeed, a bisector of \(X^\infty\) need not project under
	\[
	p_I : X(\eta_{+}, \eta_{-}) \;\cong\; \mathbb{R}^{\,n - |I|} \times X_I 
	\longrightarrow X_I
	\]
	to a bisector of \(X_I^\infty\).
	
	\medskip
	
	However, the space \(X_I \cup X_I^\infty\) is known to contain triples of ideal 
	chambers in ideal-generic position (see \cite[Proposition~5.6]{CiBars} and note that the irreducibility of the Coxeter group is not required).
	By Remark~\ref{rem::ideal_affine} these triples are also affine-generic.  
	Choose
	\[
	(\mathcal{C}_1, \mathcal{C}_2, \mathcal{C}_3) 
	\in \Ch(X_I^\infty)_{\mathrm{id\text{-}gen}}^{[3]},
	\]
	and let \(\mathcal{A}_{ij}\) be the unique affine apartment of \(X_I\) whose 
	ideal boundary contains \(\mathcal{C}_i\) and \(\mathcal{C}_j\).  
	Then for every \(i \neq j \in \{1,2,3\}\) we obtain corresponding affine 
	apartments
	\[
	A_{ij} \;\cong\; \mathbb{R}^{\,n - |I|} \times \mathcal{A}_{ij}
	\]
	of \(X(\eta_{+}, \eta_{-})\), and hence of \(X\).
	
	\medskip
	
	We now show that the apartments \(A_{12}, A_{23}, A_{13}\) contain a triple of 
	antipodal ideal chambers
	\[
	C_1,\, C_2,\, C_3 \;\in\; X^\infty
	\]
	which is \emph{not} in affine-generic position. In particular, we will obtain that 
	$$A_{12}^\infty \cap A_{23}^\infty \cap A_{13}^\infty = F^\infty.$$
	\subsubsection*{Step 1: Constructing the first two chambers}
	In the apartment \(A_{12}\), consider the bisector \(\gamma_{12}\) orthogonal 
	to a flat \(F_{12} \cong \mathbb{R}^{\,n - |I|}\) of $A_{12}$ that is parallel to $F$.  
	Let \(\mathcal H(A_{12},F_{12})\) be the finite family of affine walls in 
	\(A_{12}\) containing \(F_{12}\).  
	These walls partition \(A_{12}\) into finitely many convex cones meeting in 
	\(F_{12}\).  
	Among these, exactly two, denoted \(Q_1\) and \(Q_2\), project their ideal boundary under 
	\(p_I : A_{12} \to \mathcal A_{12}\) to the ideal chambers 
	\(\mathcal C_1\) and \(\mathcal C_2\), respectively. Notice that \(Q_1\) and \(Q_2\) are opposite as convex cones in $A_{12}$.
	
	By the symmetry of a Coxeter complex, we can assume, without loss of generality, that \(\gamma_{12} \subset Q_1 \cup Q_2\).  
	Let \(C_1\) and \(C_2\) be the ideal chambers whose barycenters are the two 
	endpoints of \(\gamma_{12}\).  
	By construction,
	\[
	C_1 \subset Q_1^\infty,
	\qquad
	C_2 \subset Q_2^\infty.
	\]
	
	\subsubsection*{Step 2: Constructing the third chamber}
	By the construction of $A_{ij}$, one notices that the intersection
	\[
	Q_1' := A_{12} \cap A_{13}
	\]
	is a convex cone parallel to \(Q_1\).  
	Thus \(\gamma_{12}\) intersects \(Q_1'\) in a geodesic ray whose ideal endpoint 
	is the barycenter of \(C_1\), and hence \(C_1 \subset (Q_1')^\infty\).
	
	Extend this ray inside \(A_{13}\) to a full bisector \(\gamma_{13}\).  
	Let \(C_3\) be the ideal chamber in \(A_{13}^\infty\) whose barycenter is the 
	endpoint of \(\gamma_{13}\) opposite to \(C_1\).  
	Then, by the construction, the pairs \((C_1,C_2)\) and \((C_1,C_3)\) are opposite.
	
	\subsubsection*{Step 3: Oppositeness of \texorpdfstring{$C_{2}$}{C2} and \texorpdfstring{$C_{3}$}{C3}}
	
	Both bisectors $\gamma_{12}$ and $\gamma_{13}$ are perpendicular to every root-flat parallel to $F_{12}$.  
	Since, by construction, the apartment $A_{23}$ contains root-flats parallel to $F$, and hence also to $F_{12}$, one might expect that extending the bisectors $\gamma_{12}$ and $\gamma_{13}$ geodesically into the apartment $A_{23}$ would force their ideal endpoints in $A_{23}^{\infty}$ to be opposite. However, this is not true in general. A counterexample already appears in the three-dimensional case when the root-flat $F$ is one-dimensional. Indeed, in this situation the orthogonal complement $F^\perp$ is two-dimensional. The geodesic extensions of the bisectors $\gamma_{12}$ and $\gamma_{13}$ into $A_{23}$ both lie in $F^\perp$, but there is no reason for them to be collinear. Consequently, the angle between these two geodesics in $F^\perp$ may be different from $180^\circ$, and therefore their ideal endpoints in $A_{23}^{\infty}$ need not be opposite.
	
	Nevertheless, if the bisector $\gamma_{12}$ in $A_{12}$ projects via $p_{I}$ to a bisector in $\mathcal{A}_{12}$, then $C_{2}$ and $C_{3}$ form an antipodal pair.
	
	Combining these observations, together with the assumption that  the bisector $\gamma_{12}$ in $A_{12}$ does project via $p_{I}$ to a bisector in $\mathcal{A}_{12}$, we conclude that the three ideal chambers
	\[
	C_{1},\; C_{2},\; C_{3}
	\]
	form a triple of antipodal chambers in $X^{\infty}$ that lies in non-affine-generic position, as required.  
	\begin{rem}
		This situation occurs, for instance, in $\widetilde{A}_{2}$--buildings, where the bisector in $A_{12}$ indeed projects to a bisector in $\mathcal{A}_{12}$. 
	\end{rem}

	\subsection{Configurations of non--ideal-generic triples of ideal chambers in affine-generic position}
	\label{subsec::conf_non_ideal_generic} We establish this under the assumption that certain unipotent subgroups contain arbitrarily many elements.

	Assume that, by the construction in Section~\ref{subsec::conf_non_affine_generic}, we have obtained an antipodal triple of ideal chambers \(C_1,\, C_2,\, C_3 \in X^{\infty}\) which is \emph{not} in affine-generic position.  As usual, for each pair \(i \neq j\) we denote by \(A_{ij}\) the unique apartment of \(X\) whose ideal boundary contains the opposite chambers \(C_i\) and \(C_j\).
	
	 In particular, by the construction in Section~\ref{subsec::conf_non_affine_generic},  there is also a root-flat $F$ of $A_{12}$ that is parallel to $A_{13}$ and $A_{23}$, and of maximal dimension among those, such that $A_{12}^\infty \cap A_{23}^\infty \cap A_{13}^\infty = F^{\infty}$. 
	
	Starting from the configuration formed by the apartments 
	\(A_{12}, A_{13}, A_{23}\), our aim is to construct new apartments \(
	A'_{12},\, A'_{13},\, A'_{23} \subset X\) together with an antipodal triple of ideal chambers \(C'_1,\, C'_2,\, C'_3 \in X^{\infty}\), such that this new triple is \emph{not} in ideal-generic position but it is in affine-generic position.
	
	Moreover, the flat \(F\) need not be of maximal dimension among those root-flats that are perpendicular to bisectors in $A_{12}$. Choose two opposite ideal simplices \(\eta_{+}, \eta_{-} \in X^{\infty}\) 
	contained in \(F^{\infty}\) and of maximal dimension in \(F^{\infty}\). Take \(X(\eta_{+}, \eta_{-}) \;\cong\; \mathbb{R}^{\,n - |I|} \times X_I\), where \(n = \dim X\), \(|I| = \dim X_I\), and the flat \(F\) is isometric to \(\mathbb{R}^{\,n - |I|}\).

	Assume that the parabolic subgroup \(P_{\eta_{+}} := \{\, g \in \Aut(X) \mid g(\eta_{+}) = \eta_{+} \,\}\)
	associated with \(\eta_{+}\) has a non--trivial unipotent part, and that this 
	unipotent part contains arbitrarily many elements. Suppose moreover that we can choose a unipotent element
	\(u \in P_{\eta_{+}}\) such that:
	\begin{enumerate}
		\item \(u\) fixes pointwise a half--apartment of \(A_{12}\) whose ideal 
		boundary contains \(C_{1}\)
		\item the wall $\mathcal{H}$ bounding this half--apartment is a facet of \(C_{1}\)
		\item $\eta_+$ is in the ideal boundary of this half-apartment, but not on the boundary of the wall $\mathcal{H}$.
	\end{enumerate}
	
	The element \(u\) acts nontrivially on \(X^{\infty}\), and in particular it 
	moves both \(\eta_{-}\) and the ideal chamber \(C_{2}\). Since 
	\(u(C_{1}) = C_{1}\) and opposition is preserved under automorphisms of \(X\), 
	we consider the apartments
	\[
	A'_{12} := A_{12}, \qquad A'_{13} := u(A_{13}),
	\]
	and observe that the pairs of ideal chambers \((C_{1}, C_{2})\) and 
	\((C_{1}, u(C_{3}))\) remain opposite. In addition, consider the intersection \(Q_{1} := A_{12} \cap u(A_{13})\).  
	Since \(F \cap \mathcal{H}\) may be a root-flat of smaller dimension than \(F\), 
	we may not have \(\partial Q_{1}^{\mathrm{fin}} \neq \emptyset\).  In this case, we will obtain an antipodal triple $\{C'_1,\, C'_2,\, C'_3\}$ such that $(A'_{12})^\infty \cap (A'_{23})^\infty \cap (A'_{13})^\infty$ is larger than the ideal boundary of the root-flat provided by Proposition~\ref{prop::non_generic_flats}.
	Nevertheless, by applying the same procedure as below and reducing the dimension of \(F \cap \mathcal{H}\) inductively, we may assume that \(\partial Q_{1}^{\mathrm{fin}} \neq \emptyset\).
	
	We claim that the pair \((C_{2},\, u(C_{3}))\) of ideal chambers is still opposite. 
	Indeed, we know by hypothesis that \((C_{2}, C_{3})\) are opposite, and that 
	\(C_{2}\) and \(u(C_{2})\) share a facet coming from the wall \(\mathcal{H}\). 
	Moreover, the gate (see \cite[Definition 3.104 and Proposition 3.105]{abramenko_brown08}) of the ideal chamber \(C_{3}\) on the wall 
	\(\mathcal{H} \cap A_{23}\), taken at \(C_{2}\), is an ideal chamber contained in the 
	ideal boundary of 
	\[
	Q_{2} := A_{23} \cap A_{12},
	\]
	and therefore remains fixed under the action of \(u\).  Recall that $Q_{2}^\infty$ contains the chamber $C_2$, and thus by \cite[Proposition 4.115]{abramenko_brown08}, $Q_{2}^\infty$ is a convex chamber subcomplex, thus is gallery connected. Moreover, by our assumptions, $\eta_+ \in Q_{2}^\infty$, but $\eta_+$ is not on the ideal boundary of the wall $\mathcal{H}$.  Therefore, there is at least one wall \(\mathcal{H} \subset A_{12}\) associated with \(C_{1}\), and hence also with \(C_{2}\), that satisfies this property.  Otherwise, one would obtain a contradiction to the assumption that \((C_{1}, C_{2}, C_{3})\) is in non--affine-generic position.

	Since \((C_{2}, C_{3})\) are opposite, and hence \((u(C_{2}), u(C_{3}))\) are also opposite in $u(A_{23})$, 
	the above observation implies that the gate of \(u(C_{3})\) with respect to 
	\(\mathcal{H}\) at \(C_{2}\) coincides with the gate of \(C_{3}\), and in particular 
	is not \(C_{2}\). By known results on the gate (see \cite[Propositions 3.103 and 3.105]{abramenko_brown08}), it follows that the pair 
	\((C_{2},\, u(C_{3}))\) is indeed opposite, as wanted. We take $A'_{23}$ to be the unique apartment in $X$ having $C_2$ and $u(C_3)$ in its ideal boundary.
	
	Since, by assumption, \(\partial Q_{1}^{\mathrm{fin}} \neq \emptyset\), the contrapositive of Proposition~\ref{prop::non_generic_flats} ensures that 
	$$C_{1}, C_{2}, u(C_{3})$$
	is an affine-generic triple of ideal chambers in \(X^\infty\). 
	
	It remains to show that $C_{1}, C_{2}, u(C_{3})$ are also in a non-ideal-generic position; more precisely, that 
	$$\eta_+ \subset (A'_{12})^\infty \cap (A'_{13})^\infty \cap (A'_{23})^\infty.$$
 Indeed, by construction, $\eta_+$ is already contained in $(A'_{12})^\infty \cap (A'_{13})^\infty$. Moreover, we have
that $\eta_+ =u(\eta_+)  \subset u(A_{23})^{\infty}$. Now, since both $C_2$ and $u(C_2)$ are opposite $u(C_3)$ and share the same gate $D$ with respect to the wall $\mathcal{H}$ at $C_2$, the intersection $u(A_{23})^{\infty} \cap A'_{23}$  is a half-apartment containing $u(C_3)$, the gate $D$, and consequently $\eta_+$ as well.

	\subsection{Summary}
	
	\begin{thm}\label{thm:summary}
		Among the irreducible finite Weyl groups of rank $\ge 2$,  the geometric condition of Proposition~\ref{prop:geom_criterion} holds exactly for the following three cases
		\[
		G_2,\qquad B_2=C_2,\qquad B_3.
		\]
		In other words, for affine buildings of type $\widetilde C_2, \widetilde G_2$ or $\widetilde B_3$, any antipodal triple of ideal chambers is ideal-generic, and hence affine-generic. 
	\end{thm}

	\begin{rem}
		Let $\Phi=\Phi_1\times\cdots\times\Phi_m$ be a direct product of irreducible finite root systems with corresponding decomposition $V=V_1\oplus\cdots\oplus V_m$. Let $\rho=(\rho_1,\dots,\rho_m)$ be the associated Weyl vector.
		
		A wall or maximal parabolic root subspace in the product is induced from a single factor. Equivalently, the corresponding line to a maximal parabolic in the product has the form $L=L_i\subset V_i$, for some $i$, and where $L_i$ is a line corresponding to a maximal parabolic of $\Phi_i$.
		
		Since every vector in $V_j$ with $j\neq i$ is automatically perpendicular to $L_i$, the perpendicularity condition from Proposition~\ref{prop:geom_criterion} reduces to the corresponding condition in the factor $V_i$. More precisely,
		\[
		\rho\in L^\perp
		\quad\Longleftrightarrow\quad
		\rho_i\in L_i^\perp.
		\]
		
		Hence the perpendicularity condition for the product root system is equivalent to the perpendicularity condition in each irreducible factor.
	\end{rem}

	\section{Algebraic case: stabilizers of triples}
	\label{sec::algebraic_case}
	
	Let now $\bfG$ be a semisimple algebraic group over a non-archimedean local field $k$. We assume that $\bfG$ is $k$-isotropic of $k$-rank $n$. We write $G:=\bfG(k)$, and denote by $X:=\Delta^{\mathrm{BT}}(\bfG,k)$ its Bruhat--Tits building and by $X^\infty$ its Tits (spherical) building at infinity. The action of $G$ on $X$ is by continuous isometries and the stabilizer of any non-empty bounded subset of $X$ is relatively compact.
	
	The goal of this section is to prove that for any antipodal triple $(C_1,C_2,C_3) $ in $\Ch(X^\infty)^{[3]}$, the $k$-rank of the stabilizer 
	$$\Stab_{G}\{C_1,C_2,C_3\} := \{g \in G \; \vert \; g(\{C_1,C_2,C_3\})= \{C_1,C_2,C_3\} \text{ setwise}\}$$
	 is equal to the dimension of the canonical root-flat $F$ attached to the triple $(C_1,C_2,C_3)$ by Proposition~\ref{prop::non_generic_flats}. 
	
	\begin{rem} Let $\mathbf{P}_1, \mathbf{P}_2, \mathbf{P}_3$ be the minimal parabolic subgroups in $G$ that stabilize the chambers $C_1, C_2, C_3$, respectively. Note that the intersection $\mathbf{P}_1 \cap \mathbf{P}_2 \cap \mathbf{P}_3$ is a subgroup of $\operatorname{Stab}_{G}\{C_1, C_2, C_3\}$. By the $k$-rank of the stabilizer $\operatorname{Stab}_{G}\{C_1, C_2, C_3\}$, we mean the $k$-rank of $\mathbf{P}_1 \cap \mathbf{P}_2 \cap \mathbf{P}_3$.

	\end{rem}

	\subsection{The canonical root-flat}
	
	We keep the notation of Section~\ref{section::opposite_ideal_ch}: for a triple $(C_1,C_2,C_3)\in\Ch(X^\infty)^{[3]}$ of pairwise opposite ideal chambers, $A_{ij}$ denotes the unique affine apartment of $X$ with $C_i,C_j\in \Ch(A_{ij}^\infty)$, and $Q_i:=A_{ij}\cap A_{ik}$ for $\{i,j,k\}=\{1,2,3\}$.
	The following proposition is a reformulation of Proposition~\ref{prop::non_generic_flats}. 
		
	\begin{prop}\label{prop:root-flat_triple}
    Let $(C_1,C_2,C_3)\in\operatorname{Ch}(X^\infty)^{[3]}$ be an antipodal triple, which we assume to be non-affine-generic. Then there exists a root-flat $F$ in $X$, well-defined up to parallelism and of maximal possible dimension $m >0$ in $A_{ij}$, for $i\neq j \in \{1,2,3\}$, such that
    \[
    F^\infty \subseteq A_{12}^\infty \cap A_{13}^\infty \cap A_{23}^\infty
    \]
    and each $A_{ij}$ contains a representative of the parallelism class of $F$. 
\end{prop}

	Notice that in the affine-generic case we have $m=0$.
	
\medskip	
	For the rest of the section we let $(C_1,C_2,C_3)\in\Ch(X^\infty)^{[3]}$ be an antipodal, non-affine-generic triple, with associated root-flat $F$ and $m=\dim F > 0$. Fix two opposite ideal simplices $\eta_+,\eta_-\in F^\infty$ of maximal dimension; thus $\dim\eta_+=\dim\eta_-=m-1$. With this choice Proposition~\ref{prop::res_building} yields the splitting
	\[
	X(F) \;=\; X(\eta_+,\eta_-) \;\cong\; \bbR^{n-|I|}\times X_I,
	\]
	where $n-|I|=m$, the Euclidean factor identifies with the parallelism class of $F$, and $X_I$ is the transverse locally finite affine building of dimension $n-m$. Write $p_I:X(F)\to X_I$ for the canonical projection and denote the projected triple $(p_I(C_1), p_I(C_2),p_I(C_3))$ by $(\mathcal C_1,\mathcal C_2,\mathcal C_3)\in\Ch(X_I^\infty)^{[3]}$.
	
	\subsection{The convex hull for non-affine-generic triples}
	We maintain the notation and hypotheses established in the previous subsection, but now $m=\dim F \ge 0$.
		
	Definition~\ref{def::generic_position_convex_hull} introduced the bounded convex hull $\Gamma(C_1,C_2,C_3)$ only for affine-generic triples (where each $\partial Q_i^{\mathrm{fin}}\ne\emptyset$). In this section we extend it to non-affine generic triples. 
	
	\begin{lem}\label{lem:hull}
		The projected triple $(\mathcal C_1,\mathcal C_2,\mathcal C_3)$ in $X_I$ is affine-generic in $X_I$. 
	\end{lem}

	\begin{proof}
		By Lemma~\ref{lem::opposition_projection}\,(2)--(3) the chambers $\mathcal C_i$ form an antipodal triple in $X_I^\infty$. Going through the proof of Proposition~\ref{prop::non_generic_flats} with $F_1=F$ playing the role of the maximal-dimension root-flat in $\partial Q_1^{\mathrm{inf}}$, we obtain that $\partial\mathcal{Q}_i^{\mathrm{fin}}\ne\emptyset$ for every $i\in\{1,2,3\}$, so $(\mathcal C_1,\mathcal C_2,\mathcal C_3)$ is affine-generic in $X_I$. 
	\end{proof}
	
	\begin{Def}\label{def:hull_extended}
		The convex hull $\Gamma_{X_I}(\mathcal C_1,\mathcal C_2,\mathcal C_3)\subset X_I$ in the affine building $X_I$ is well-defined by Definition~\ref{def::generic_position_convex_hull} and Lemma~\ref{lem:hull}. We set
		\[
		\Gamma(C_1,C_2,C_3) \;:=\; F\times\Gamma_{X_I}(\mathcal C_1,\mathcal C_2,\mathcal C_3) \;\subset\; X(F)\subset X,\]
		using the splitting $X(F)\cong F\times X_I$.
	\end{Def}
	
	Observe that for $m=0$ Definition \ref{def:hull_extended} reduces to Definition~\ref{def::generic_position_convex_hull}.
	
Notice that in Definition~\ref{def:hull_extended}, \(F\) is needed only up to its parallelism class in \(X\).
	
	\subsection{Levi structure}
		
	Let $\mathbf{P}_{\eta_+}, \mathbf{P}_{\eta_-} \subset \mathbf{G}$ denote the parabolic $k$-subgroups stabilizing $\eta_+$ and $\eta_-$, respectively. They are opposite parabolic subgroups whose intersection is a Levi $k$-subgroup of either. We denote this Levi subgroup by $\mathbf{L}_{\eta_+}$ and its group of $k$-points by $L_{\eta_+} = \mathbf{L}_{\eta_+}(k)$. Let $\mathbf{S}$ denote the maximal $k$-split torus of the center $Z(\mathbf{L}_{\eta_+})$. It is well known that
\[
\mathbf{L}_{\eta_+} = Z_{\mathbf{G}}(\mathbf{S}) = \mathbf{P}_{\eta_+} \cap \mathbf{P}_{\eta_-}.
\]
Moreover, $\mathbf{S}$ has $k$-rank $m = \dim F$, and the action of $\mathbf{S}(k)$ on $X$ preserves the apartment $A_{ij}$, for $i\neq j \in \{1,2,3\}$, acting on $F$ by translations along directions in $F$.

We write $\mathbf{L}_{\eta_+} = \mathbf{S} \cdot \mathbf{M}$, where $\mathbf{M} \subset \mathbf{L}_{\eta_+}$ is the derived (semisimple) subgroup of $\mathbf{L}_{\eta_+}$. This is an almost direct product, $\mathbf{L}_{\eta_+} \cong (\mathbf{S} \times \mathbf{M}) / Z_0$, with $Z_0 := \mathbf{S} \cap \mathbf{M}$ being a finite central subgroup. The group of $k$-points of the semisimple $k$-group $\mathbf{M}$ acts faithfully on the transverse affine building $X_I$, while $\mathbf{S}(k)$ acts trivially on $X_I$.
		
	\begin{ex}\label{ex:SL3}
    Take $\mathbf{G}=\operatorname{SL}_3$ over $\mathbb{Q}_p$ (whose Bruhat--Tits building is of affine type $\widetilde{A}_2$). Three pairwise opposite full flags $(C_1,C_2,C_3)$ in the spherical building can fail to be affine-generic. Indeed, up to the $\operatorname{SL}_3(\mathbb{Q}_p)$-action, one may have that $F$ is the geodesic line (which in this case is a wall) in the standard apartment whose ideal endpoints are the barycenters of two opposite \emph{panels} (codimension-1 simplices) $\eta_+,\eta_-\in X^\infty$ (which in this case are ideal vertices). The ideal simplex $\eta_+$ corresponds to a single proper subspace $W\subset\mathbb{Q}_p^3$ (a line or a plane). Take, for instance, $W=\mathbb{Q}_p e_2$, so that
    \[
    P_{\eta_+} = \operatorname{Stab}_{\operatorname{SL}_3}(\mathbb{Q}_p e_2) = \begin{pmatrix} * & 0 & * \\ 0 & * & 0 \\ * & 0 & * \end{pmatrix} \cdot U_{\eta_+},
    \]
    which is a maximal parabolic subgroup with Levi factor $\mathbf{L}_{\eta_+} \cong \operatorname{GL}_2$ (acting on the $(e_1,e_3)$-plane), embedded in $\operatorname{SL}_3$ via $A \mapsto \operatorname{diag}(A_{11}, (\det A)^{-1}, A_{22})$ in suitable coordinates. The central torus $\mathbf{S}=Z(\mathbf{L}_{\eta_+})^\circ \cong \mathbb{G}_m$ has rank $1=m$, and its group of $\mathbb{Q}_p$-points $\mathbb{Q}_p^\times$ acts on $X$ by translations along $F$. The derived subgroup $\mathbf{M}=[\mathbf{L}_{\eta_+},\mathbf{L}_{\eta_+}] \cong \operatorname{SL}_2$ acts on the transverse building $X_I$, which is the Bruhat--Tits tree of $\operatorname{SL}_2(\mathbb{Q}_p)$.
\end{ex}

	The next lemma is straightforward, but we include a proof for completeness. 
	\begin{lem}\label{lem:reductions}
		The set-wise stabilizer $H:=\Stab_G\{C_1,C_2,C_3\}$ preserves the parallel class of $F$, hence
		\[
		H \;\subseteq\; \Stab_G(F^\infty) \;=\; L_{\eta_+}.
		\]
		Moreover the projection of $H$ to the quotient $L_{\eta_+}/\bfS(k)\cong\bfM(k)/Z_0(k)$ stabilizes the affine-generic triple $(\mathcal C_1,\mathcal C_2,\mathcal C_3)\in\Ch(X_I^\infty)^{[3]}$ in $X_I^\infty$, and is therefore contained in a compact subgroup of $\bfM(k)$.
	\end{lem}
	
	\begin{proof}		
		By Proposition~\ref{prop:root-flat_triple}, the parallelism class $[F]$ depends only on the unordered triple $\{C_1,C_2,C_3\}$, so any $g\in H$ stabilizes the unordered pair $\{\eta_+,\eta_-\}$ at infinity. If $g$ swaps $\eta_+$ and $\eta_-$, then $g^2$ preserves both. Thus, up to passing to a subgroup $H_0\subseteq H$ of index at most 2, we may assume that each element of $H_0$ stabilizes $\eta_+$ (and hence $\eta_-$). The fixator $\operatorname{Fix}_G(\eta_+,\eta_-)$ is equal to $P_{\eta_+}\cap P_{\eta_-}=L_{\eta_+}$, which implies $H_0\subseteq L_{\eta_+}$.

The image of $L_{\eta_+}$ under the natural quotient map $L_{\eta_+}\to L_{\eta_+}/\mathbf{S}(k)$ acts on the transverse building $X_I$ and stabilizes the projected unordered triple $\{\mathcal{C}_1,\mathcal{C}_2,\mathcal{C}_3\}$, which is affine-generic in $X_I$ by Lemma~\ref{lem:hull}. By Proposition~\ref{prop::barycenter_map} (applied to this affine-generic triple in $X_I$), the barycenter $\zeta_{X_I}(\mathcal{C}_1,\mathcal{C}_2,\mathcal{C}_3)\in X_I$ is canonically attached to the projected triple and is therefore fixed by the projection of $H_0$. Because the stabilizer of a point in $X_I$ under the action of $\mathbf{M}(k)$ is compact, the projection of $H_0$ to $\mathbf{M}(k)$ lies in a compact subgroup. Since the quotient group $H/H_0$ is finite, the same conclusion holds for $H$.
	\end{proof}
	
	\subsection{The rank equality}
	
	\begin{thm}\label{thm:rank}
		Let $(C_1,C_2,C_3)\in\Ch(X^\infty)^{[3]}$ be an antipodal triple, with canonical root-flat $F$ of dimension $\dim F=m \geq 0$ given by Proposition~\ref{prop:root-flat_triple}. Then
		\[
		\rk_k \bigl(\Stab_{G}\{C_1,C_2,C_3\}\bigr) \;=\; m.
		\]
		More precisely, there exists a finite index subgroup $H_0$ of $H:= \Stab_{G}\{C_1,C_2,C_3\}$, and a short exact sequence
		\[
		1 \;\longrightarrow\; K \;\longrightarrow\; H_0 \;\longrightarrow\; \Lambda \;\longrightarrow\; 1
		\]
		with $K$ compact and $\Lambda$ a free abelian group of rank $m$ contained in $\bfS(k)$, acting on $F$ by translations.
	\end{thm}
	
	\begin{proof}
		If $m=0$ the triple is affine-generic and $\Stab_G\{C_1,C_2,C_3\}$ fixes the barycenter $\zeta(C_1,C_2,C_3)\in X$ (Proposition~\ref{prop::barycenter_map}), hence is compact. In this case the statement holds with $K=\Stab_G\{C_1,C_2,C_3\}$ and $\Lambda=1$. 
		
		Assume now that $m>0$, and write $H:=\operatorname{Stab}_G\{C_1,C_2,C_3\}$. As in the proof of Lemma~\ref{lem:reductions}, replace $H$ by its subgroup $H_0$ of index at most 2 that stabilizes $\eta_+$, so that $H_0\subseteq L_{\eta_+}$.
		
		The Levi $\bfL_{\eta_+}$ acts on $F\cong\bbR^m$ through its central torus $\bfS$ via the cocharacter map $X_*(\bfS)\otimes\bbR\xrightarrow{\sim}F$ by translations, and the derived subgroup $\bfM$ acts trivially on $F$. Restricting the action of $H_0$ on $X(F)$ to the $F$-factor yields a homomorphism
\[
\phi \colon H_0 \longrightarrow \operatorname{Trans}(F)
\]
that maps each element $h \in H_0$ to the translation it induces on $F$. By construction, $\phi$ factors through the projection sequence $L_{\eta_+} \to L_{\eta_+} / \mathbf{M}(k) \to \mathbf{S}(k) / (\mathbf{S} \cap \mathbf{M})(k)$, as well as through the canonical valuation $\mathbf{S}(k) \to X_*(\mathbf{S}) \otimes \mathbb{R} \cong F$.

We claim that $\Lambda := \phi(H_0)$ is a discrete subgroup of $F$. Indeed, let us choose a vertex $x_0 \in F \subset X$. Since $X$ is simplicial and locally finite, the $H_0$-orbit of $x_0$ in $X$ is discrete. Projecting this orbit to the $F$-factor via the decomposition $X(F) \cong F \times X_I$, we see that the orbit $\phi(H_0) \cdot x_0 \subset F$ is also discrete. Because $\phi$ acts on $F$ by translations, $\phi(H_0)$ acts simply on this orbit. It follows that $\phi(H_0) \cong \phi(H_0) \cdot x_0$ is discrete in $F \cong \mathbb{R}^m$, meaning it is a free abelian lattice of rank at most $m$.
		
		The torus $\bfS$ is contained in the Levi of the parabolic subgroups $P_{\eta_+} $ and $P_{\eta_-}$ and in particular in any maximal split torus of $\bfG$ stabilising $A_{12}$ (since $A_{12}\subset X(F)$). Translations of $A_{12}$ fix every chamber of $A_{12}^\infty$, so $\bfS(k)$ fixes $C_1$ and $C_2$. For the same reason $\bfS(k)$ fixes every $C_3$ as well. Hence $\bfS(k)\subseteq H_0$. Inside $\bfS(k)$, consider the value-group homomorphism
		\[
		\mathrm{val}_{\bfS}:\;\bfS(k)\;\xrightarrow{\;\cong\;}\;(k^\times)^m\;\xrightarrow{\;(v_k,\dots,v_k)\;}\;\bbZ^m,
		\]
		a continuous surjection with compact kernel $(\mathcal O_k^\times)^m\subset\bfS(k)$. As seen above, $\bfS(k)$ acts on $A_{12}$ by translations, and on the $F$-factor of $X(F)$ via $\mathrm{val}_{\bfS}$ composed with a $\bbZ$-linear isomorphism $\bbZ^m\hookrightarrow F$. As $\bfS(k)\subseteq H_0$, we obtain that $\phi(\bfS(k))$ is contained in $\Lambda$, so that $\Lambda$ has rank $\ge m$ and therefore $\rk(\Lambda)= m$.
		
		Let us define $K:=\ker\phi$. By definition the image of $K$ in $\bfS(k)/(\bfS\cap\bfM)(k)$ lies in the compact part $(\mathcal O_k^\times)^m/(Z_0\cap(\mathcal O_k^\times)^m)$. Composing with the projection $L_{\eta_+}\to L_{\eta_+}/\bfS(k)$, the image of $K$ stabilizes the affine-generic triple $(\mathcal C_1,\mathcal C_2,\mathcal C_3)$ in $X_I$ and is therefore compact (Lemma~\ref{lem:reductions}). Both factors being compact, $K$ is itself compact.
		
		Putting everything together, we obtain the announced short exact sequence and the rank equality. The conclusion holds for $H$ since $[H:H_0]\le 2$.
	\end{proof}
	
	\subsection{Extremal cases and examples}
	
	In view of Theorem~\ref{thm:rank}, affine-genericity has the following consequences. 
	
	\begin{cor}[Affine-generic case]\label{cor:m0}
		If $(C_1,C_2,C_3)$ is affine-generic, then $m=0$, $\Lambda$ is trivial and $\Stab_G\{C_1,C_2,C_3\}$ is compact.
	\end{cor}
	\begin{rem}
	\label{rem::extremal_case}
		The other extreme case is when the three apartments $A_{ij}$ in $X$ share a root-flat that is a wall (a parallelism class of codimension-$1$ flats). Then $\eta_+$ is a panel of $\bdinf$ and $P_{\eta_+}$ is a parabolic subgroup with Levi $\bfL_{\eta_+}=\bfS\cdot\bfM$ where $\bfS$ has $k$-rank $n-1$ and $\bfM$ is anisotropic over $k$. Consequently
		\[
		\Stab_G\{C_1,C_2,C_3\}/K \cong \Lambda\;\subseteq\;\bfS(k)\cong (k^\times)^{n-1}
		\]
		has rank $n-1$. This situation can happen for $G=\SL_3(k)$. 
	\end{rem}
	
	In view of Theorem~\ref{thm:summary}, we immediately obtain the following. 
	\begin{cor}[Automatic compactness]\label{cor:auto}
		If the affine type of $\bfG(k)$ is $\widetilde C_2$, $\widetilde G_2$ or $\widetilde B_3$, then by Theorem~\ref{thm:summary} every antipodal triple in $\Ch(X^\infty)^{[3]}$ is ideal-generic (hence affine-generic). Consequently the stabilizer in $G$ of every triple of pairwise opposite chambers of $X^\infty$ is compact.
	\end{cor}
	
	\begin{rem}\label{rem:rank_function}
		The list in Corollary~\ref{cor:auto} includes the split groups $\mathrm{Sp}_4,\mathrm{SO}_5,\mathrm{Spin}_5$ (type $\widetilde C_2$), the exceptional group $G_2$, and $\mathrm{Spin}_7,\mathrm{SO}_7$ (of type $\widetilde B_3$. The same conclusion holds for any group whose relative root system over $k$ has these absolute types. In all other simple types the geometric criterion of Proposition~\ref{prop:geom_criterion} can produce antipodal triples $\{C_1,C_2,C_3\}$ with non-trivial canonical root-flat $F$. 
	\end{rem}
	
	\begin{rem}[$\SL_3$ revisited]\label{rem:SL3}
		Specializing to $\bfG=\SL_3$ over $\bbQ_p$ (Example~\ref{ex:SL3}): $n=2$, so $m\in\{0,1\}$.
		\begin{itemize}
			\item $m=0$: the triple is affine-generic, the stabilizer is compact.
			\item $m=1$: the triple is maximally non-generic (see Remark \ref{rem::extremal_case}). The canonical root-flat is a line $F\subset A_{12}$, the central torus $\bfS\cong\bbG_m$ has $\bbQ_p$-points $\bbQ_p^\times$, and
			\[
			\Stab_{\SL_3(\bbQ_p)}\{C_1,C_2,C_3\}/K \;\cong\; \bbZ
			\]
			is the value group of $\bbQ_p^\times$, acting on $F$ by integer translations.
		\end{itemize}
		Note that the maximally non-generic case is realised, e.g., by triples produced by the construction of \S\ref{subsec::conf_non_affine_generic}: starting from an affine-generic triple in the transverse tree $X_I$ and lifting via $X(F)\cong F\times X_I$, one obtains three opposite chambers in $X^\infty$ whose stabilizer in $\SL_3(\bbQ_p)$ is $\mathcal O_{\bbQ_p}^\times\cdot\bbZ\cong\bbZ_p^\times\rtimes\bbZ$, all sitting inside $\bbQ_p^\times\subset\GL_2(\bbQ_p)\subset\SL_3(\bbQ_p)$.
	\end{rem}

	\begin{ackn}
		C.C is supported by a research grant (VIL53023) from VILLUM FONDEN.  C.LB is supported by ERC Advanced Grant NET 101141693.  The calculations in Sections \ref{subsec::not_perp_bisectors} and \ref{subsec::perp_bisectors} were carried out with the assistance of Copilot. 
	\end{ackn}

	\bibliographystyle{alpha}
	\bibliography{bibliography}	
\end{document}